\documentclass[11pt]{amsart}

\usepackage{amsthm}
\usepackage{amsbsy}
\usepackage{amssymb}
\usepackage{amsfonts}
\usepackage{amsmath}
\usepackage{graphics}
\usepackage[dvips]{lscape} 
\usepackage{verbatim}
\usepackage{epstopdf}
\usepackage[pdftex]{graphicx}

\theoremstyle{plain}
\newtheorem{thm}{Theorem}[section]

\newtheorem{prop}[thm]{Proposition}
\newtheorem{lemma}[thm]{Lemma}
\newtheorem{coro}[thm]{Corollary}

\long\def\symbolfootnote[#1]#2{\begingroup%
\def\thefootnote{\fnsymbol{footnote}}\footnote[#1]{#2}\endgroup}

\title{Quantitative spectral gap for thin groups of hyperbolic isometries}
\author{Michael Magee}
\address{Department of Mathematics, University of California at Santa Cruz, 1156 High Street, Santa Cruz, CA 95064.}
\email{mmagee@ucsc.edu}
    
\thanks{The author was supported in part by DARPA via AFOSR grant FA9550-08-1-0315 and by the University of California via the Chancellor's Fellowship.}
\begin{document}
\begin{abstract}
Let $\Lambda$ be a subgroup of an arithmetic lattice in $\mathrm{SO}(n+1 , 1)$. The quotient $\mathbb{H}^{n+1} / \Lambda$ has a natural family of congruence covers corresponding to ideals in a ring of integers. We establish a super-strong approximation result for Zariski-dense $\Lambda$ with some additional regularity and thickness properties. Concretely, this asserts a quantitative spectral gap for the Laplacian operators on the congruence covers. This generalizes results of Sarnak and Xue (1991) and Gamburd (2002).
\end{abstract}
\maketitle

\renewcommand{\baselinestretch}{1.5}

\def\Z{\mathbb{Z}}
\def\FZ{\mathbb{F}_2^{\Z}}
\def\CFZ{\mathbb{C} \FZ }
\def\CS{C$^*$}
\def\Hom{\mathrm{Hom}}
\def\IM{\mathrm{Im}}
\def\R{\mathbb{R}}
\def\Z{\mathbf{Z}}
\def\C{\mathbb{C}}
\def\VNA{von Neumann algebra}
\def\B{\mathcal{B}}
\def\G{\Gamma}
\def\SLT{\mathrm{SL}(2,\mathbb{R})}
\def\SL3{\mathrm{SL}(3,\mathbb{R})}
\def\SLN{\mathrm{SL}(n,\mathbb{R})}
\def\USV{U(\sigma,v)}
\def\a{\alpha}
\def\b{\beta}
\def\M{\mathcal{M}(G)}
\def\g{\gamma}
\def\CIC{C_c^\infty}
\def\CP1{\mathbb{C}P^1} 
\def\CP13{\mathbb{C}P^1 \times \mathbb{C}P^1 \times \mathbb{C}P^1}
\def\Ind{\mathrm{Ind}}
\def\Res{\mathrm{Res}}
\def\End{\mathrm{End}}
\def\Dim{\mathrm{Dim}}
\def\d{\delta}
\def\e{\epsilon}
\def\s{\sigma}
\def\Isom{\mathrm{Isom}}
\def\H{\mathbb{H}}
\def\adj{\mathrm{adj}}
\def\tchi{\tilde{\chi}}
\def\N{\mathbb{N}}
\def\L{\Lambda}
\def\P{\mathcal{P}}
\def\Spin{\mathrm{Spin}}
\def\K{\mathcal{K}}
\def\GG{\mathcal{G}}
\def\SO{\mathrm{SO}}
\def\O{\mathcal{O}}
\def\SL{\mathrm{SL}}
\def\Q{\mathbf{Q}}
\def\F{\mathcal{F}}
\def\vp{\varphi}

\def\RR{\mathcal{I}}
\def\lg{\mathfrak{g}}
\def\Aut{\mathrm{Aut}}
\def\Tr{\mathrm{Tr}}


\section{Introduction} 
Let $n > 1$, $F$ a totally real number field with a fixed infinite place, and  $\GG = \SO(F^{n+2} , q)$ the closed $F$-subgroup of $\mathrm{GL}_{n+2}$ which preserves a quadratic form $q$ defined over $F$. We require $\GG(\R) \cong \SO (n+1,1)$ at the fixed place and compact at the other real places. Let $\O_F$ be the ring of integers of $F$ and $\Gamma = \mathcal{G}(F) \cap \mathrm{GL}_{n+2}(\O_F)$. Ideals $\RR$ in $\O_F$ give a level structure by defining
\begin{equation}
\Gamma(\RR) = \{ \g \in \G : \g \equiv I \mod \RR \}
\end{equation}
the \textit{principal congruence subgroup} at level $\RR$. Let $\Lambda$ be a subgroup of $\G$ which is Zariski-dense in $\GG(\C)$. We also assume that the traces of $\Lambda$ in the adjoint representation generate the ring $\O_F$. Then $\Lambda$ inherits a level structure by defining
\begin{equation}
\Lambda(\RR) = \Lambda \cap \Gamma(\RR).
\end{equation}
The groups $\G$, $\Lambda$, $\Lambda(\RR)$ act by isometries on hyperbolic space $\H^{n+1}$. On one hand, $\Lambda$ can be \textit{thin} in $\Gamma$ as it is possibly infinite index, but on the other hand it is thick enough (Zariski-dense with large trace field) so that the group structure does not degenerate. We further require that $\Lambda$ is geometrically finite. This means that any Dirichlet fundamental domain for $\Lambda$ in $\H^{n+1}$ is finitely faced, and implies that $\Lambda$ is finitely generated\footnote{We refer the reader to \cite{KAP00} for general details about hyperbolic geometry.}.

Adding this geometric regularity gives us a further gauge of thickness. For a point $o \in \H^{n+1}$ the orbit $\Lambda o $ accumulates on a subset of the boundary $S^{n}_\infty$ of $\H^{n+1}$. This is called the limit set of $\Lambda$ and denoted $L(\Lambda)$. This set is Cantor-like and has an associated Hausdorff dimension $\delta(L(\Lambda))$. The motif of this paper is that if $\delta(L(\Lambda))$ is large enough, that is we have 'sufficient thickness', then there is a 'super-strong' approximation statement for the congruence quotients $\Lambda / \Lambda(\RR)$. Concretely, this asserts the existence of a spectral gap.

The congruence quotients $\Lambda / \Lambda(\RR)$ act as deck transformations on the quotient space $X(\RR) \equiv \H^{n+1} / \Lambda(\RR)$ and induce a locally isometric covering
\begin{equation}
\pi_\RR : X(\RR) \to X \equiv \H^{n+1} / \Lambda.
\end{equation}
We make some extra assumptions on $\Lambda$ to ease exposition. By Selberg's Lemma \cite{SEL60}, which states that any finitely generated matrix group has a normal subgroup of finite index without torsion, we can pass to a finite index normal subgroup without elliptic or orientation reversing elements. In addition we may need to pass to a subgroup which is the kernel of the spinor norm at a localization of $\O_F$, we deal with this subtlety in Section \ref{algebra}. Any analysis will then occur on finite coverings of the initial $X$ and $X(\RR)$. 

It follows in Section \ref{algebra} from the work of Weisfeiler \cite{WEI84} that away from finitely many primes the natural inclusion $\Lambda / \Lambda(\RR) \hookrightarrow \G / \G(\RR)$ is \textit{onto} some large subgroup\footnote{The kernel of a spinor norm at a finite semi-local ring.} $(\G / \G(\RR) )'$. In other words, if we choose a set of generators  $S = \{ A_1 , \ldots , A_k \}$ for $\Lambda$ and consider the Cayley graph $\mathcal{H}_\RR \equiv \mathcal{H} (  S\G(\RR) , (\G / \G(\RR) )' )$ then this graph is \textit{connected}. This is the initial strong approximation statement which is to be strengthened.

A natural such strengthening is to insist that these graphs be 'highly connected', a concept which can be made precise by defining the expansion coefficient of a $k$-regular graph $\mathcal{H}$ 
\begin{equation}
c( \mathcal{H} ) = \inf \left\{ \: \frac{ | N(W) | }{ |W| } \: : \: |W| < \frac{1}{2} |\mathcal{H}| \: \right\}
\end{equation}
where $W$ runs over subsets of the vertex set of $\mathcal{H}$ and $N(W)$ denotes its set of neighbours. The Cayley graphs in question have associated discrete Laplacians which are neighbour-averaging operators on functions on the vertices. As in \cite{LU93} the spectral theory of these operators is related to the expansion coefficients of the graphs. The highly-connectedness will then be a property of a family of graphs (in our case $\mathcal{H}_\RR$), we say that $\mathcal{H}_\RR$ are a family of expanders if there exists $C$ real such that
\begin{equation}
\liminf_{|\O_F / \RR| \to \infty} c(\mathcal{H}_\RR ) \geq C > 0 .
\end{equation}
Suitably reinterpreted this asserts a spectral gap for the graph Laplacians. It is a fairly direct consequence\footnote{This argument appears in detail in \cite{GAM02}.} of Fell's continuity of induction \cite{FEL62} that by passing through representation theoretic descriptions of the action of the Laplacians (graph theoretic and diffeo-geometric), it is sufficient to prove a spectral gap result 'up above' to establish the expansion property for the Cayley graphs $\mathcal{H}_\RR$. To summarize, the spectral gap for the manifolds $X(\RR)$ would imply a connectedness property which naturally goes beyond strong approximation, hence 'super-strong' approximation. 

To see the asserted spectral gap, we consider the Laplacian operators $\Delta_X$, $\Delta_{X(\RR)}$, and in particular, their $L^2$ spectra denoted $\Omega(X)$, $\Omega(X(\RR))$ respectively. The bottom of the spectrum was characterized by Sullivan in \cite{SUL82}, and Lax and Phillips established a finiteness property in \cite{LP82}. These results can be summed up as follows.
\newpage
\begin{thm}[Sullivan, Lax and Phillips] Suppose $\d > n/2$. The following hold:
\begin{enumerate}
\item The bottom of the $L^2$ spectrum of the Laplacian on $X$ (resp. $X(\RR)$) is an eigenvalue of multiplicity one at $\lambda_0 = \delta( n - \delta )$.
\item The $L^2$ spectrum in the range $[ \delta( n - \delta ) , n^2  / 4 )$ consists of finitely many discrete eigenvalues.
\end{enumerate}
\end{thm}
The hypothesis of our main theorem will imply $\delta > n/2$. 

The Hausdorff dimension of the limit set $L(X(\RR))$ is the same for all $\RR$, so this result tells us the bottom of the spectrum at all levels. Moreover any eigenfunction of $\Delta_{X}$ lifts to an eigenfunction of $\Delta_{X(\RR)}$ with the same eigenvalue, so by the finiteness statement in the previous theorem we know that at level $\RR$ the discrete spectrum of $\Delta_{X(\RR)}$ in $(0,n^2 /4)$ consists of that of $\Delta_X$ in addition to finitely many new eigenvalues. Our main theorem gives an explicit range in which there can be no new eigenvalues (for $\RR$ avoiding finitely many primes). Together with the finiteness of the spectrum in $(0,n^2/4)$ for $X$ and at each of the finitely many excluded levels, this implies the existence of a spectral gap.

Before stating our main theorem we give some history of the spectral gap. Our starting point is Selberg's seminal paper \cite{SEL65} where it is proved
\begin{thm}[Selberg]\label{316}
Let $\G(N)$ be the principal congruence subgroup of $\SL_2(\Z)$ at level $N$, and $\G' \supset \G(N)$. Then letting $X'(N) = \H^2 / \G'$ and writing $\lambda_1(X'(N))$ for the first non-zero eigenvalue of $\Delta_{X'(N)}$ we have for any $N \geq 1$
\begin{equation}\label{sel1}
\lambda_1(X'(N)) \geq 3/16 . 
\end{equation}
\end{thm}
It was conjectured by Selberg at the same time that in fact, with notation as before
\begin{equation}
\lambda_1(X'(N)) \geq 1/4 .
\end{equation} 
Selberg's $3/16$ result is no longer the state of the art\footnote{Luo, Rudnick and Sarnak in \cite{LRS95} proved $\lambda_1(X'(N)) \geq 171/784$ by using properties of $\mathrm{GL}_3$ Rankin-Selberg $L$-functions. Shortly after this Iwaniec \cite{IWA96} proved the slightly weaker $\lambda_1(X'(N)) \geq 10/49$ by using only the $\mathrm{GL}_2$ theory. Kim and Shahidi \cite{KS02} proved $\lambda_1(X'(N)) \geq 66/289$ via the existence of the functorial symmetric cube for $\mathrm{GL}_2$. After some further developments along these lines (functorial powers) Kim and Sarnak \cite{KS03} proved the current best result, which says $\lambda_1(X'(N)) \geq 975/4096$.
}, however, the conjectured $1/4$ remains unattained. This is a fundamental open problem of modular forms. The reader can read Sarnak's notice \cite{SAR95} for a friendly exposition of the subject, as well as the notes of Sarnak \cite{SAR05} for more recent developments.
 
Theorem \ref{316} was generalized to three dimensions by Sarnak \cite{SAR83} where it was proved that if $E$ is any quadratic imaginary number field and $\O_E$ the ring of integers then
\begin{equation}
\lambda_1( \H^3 / \SL_2(\O_E) ) \geq 3/4 .
\end{equation}
The method used there is very much in the spirit of Selberg's proof of \eqref{sel1}. It extends to congruence subgroups of $\SL_2(\O_E)$. This was further extended to arbitrary dimension by Elstrodt, Grunewald and Mennicke \cite{EGM90} and Cogdell, Li, Piatetski-Shapiro and Sarnak \cite{CLPS91} independently. Both these papers prove that if $Q$ is a quadratic form of signature $(1 , n + 1)$, $n > 1$, $Q$ isotropic over $\Q$, and $\G$ a congruence subgroup of $\SO^0_{n+2} (\Z, Q)$ then
 \begin{equation}
 \lambda_1( \H^{n+1} / \G ) \geq \frac{2n - 1}{4}.
 \end{equation}
Work of Burger and Sarnak \cite{BS91} gave further progress by allowing one to link the Laplacian spectrum of conguence hyperbolic manifolds to the automorphic spectrum of $\mathrm{GL}_2$. Together with results of Blomer and Brumley \cite{BB11} this lifting argument yields that when $\GG$ and $\G$ are as in our setup and $\G(\RR)$ is a congruence subgroup,
\begin{equation}
\lambda_1( \H^{n+1} / \G(\RR)) > \frac{25}{32}\left( n - \frac{25}{32} \right) .
\end{equation}

Recently, following the proof of the fundamental lemma by Ng\^{o} \cite{NGO10} and the weighted fundamental lemma by Chaudouard and Laumon \cite{CL1,CL2}, certain conditional results of Arthur appearing in \cite{A05} have become fact. Using these results of Arthur, Bergeron and Clozel proved in \cite{BC12} the following result on the spectrum of the Laplacian.

\begin{thm}[Bergeron, Clozel]\label{BC}
Let $\GG$ be a $\Q$-group obtained by restriction of scalars from a special orthogonal group (split or quasi-split) over a totally real number field. Additionally suppose that $\GG$ does not come from a twisted form $^{3}D_4$ or  $^{6}D_4$ and that $\GG(\R)$ is the product of $\SO(n+1 , 1)$ with a compact group. For any torsion free congruence subgroup $\G \subset \GG$, the spectrum of the Laplacian on $\H^{n+1} / \G$ is contained in the set
\begin{equation}
\bigcup_{0 \leq j < \frac{n}{2} } \{ j ( n - j) \} \cup \left[ \frac{n^2}{4} - \left( \frac{1}{2} - \frac{1}{N^2 + 1} \right)^2 , \infty \right) , 
\end{equation}
where $N = n + 1$ if $n$ is odd and $N = n+2$ if $n$ is even.
In particular, when $n \geq 3$ we have the spectral gap result
\begin{equation}
\lambda_1( \H^{n+1} / \G) \geq n - 1.
\end{equation}
\end{thm}

This result feeds our main theorem via the work of Kelmer and Silberman \cite{KS10} relating the spectral theory to the lattice point count. When $n \geq 3$ Theorem \ref{BC} gives the best possible input for our method. When $n = 2$, we use the lattice point of Sarnak and Xue \cite{SX91} which is also the best possible.
 
The proof of Theorem \ref{316}, for example, relies essentially on the underlying arithmetic of the modular group $\SL_2(\Z)$ and associated Kloosterman sums\footnote{The estimates for Kloosterman sums come from Weil \cite{WEIL48} and appeal to the Riemann hypothesis for curves (also proved by Weil). Iwaniec proved \eqref{sel1} in \cite{IWA89} using Kloosterman sums but without relying on Weil's bound. Gelbart and Jacquet \cite{GJ78} proved that $3/16$ is not attained in \eqref{sel1} by a very different method to that of Selberg.}. We will not have access to such rich arithmetic and will rely on a more robust 'almost geometric' method developed by Sarnak and Xue \cite{SX91} and extended to the two dimensional infinite volume case by Gamburd \cite{GAM02}. Sarnak and Xue proved
\begin{thm}[Sarnak, Xue] Let $F$ be a totally real number field with a fixed infinite place, and $\O_F$ the ring of integers of $F$. Let $\GG$ be an orthogonal $F$-subgroup of $\mathrm{GL}_{4}$ with $\GG(\R) \cong \SL_2(\R)$ (resp. $\SL_2(\C)$) at the fixed place and compact at the other real places. Let $\G$ be a finite index subgroup of $\GG(F) \cap \mathrm{GL}_{4}(\O_F)$ which is cocompact in $\SL_2(\R)$ (resp. $\SL_2(\C)$). Then for large enough prime ideals $\P \subset \O_F$
\begin{equation*}
\Omega(  \H^{n+1}  / \G (\P) ) \cap [0, \mu ) = \Omega(   \H^{n+1} / \G (1) ) \cap [0, \mu)
\end{equation*}  
where $n = 1$ (resp. $2$) and $\mu = 5/36$ (resp. $11/36$) in the case of $\SL_2(\R)$ (resp. $\SL_2(\C)$).
\end{thm}
The Sarnak and Xue machine makes use of the fact that if new eigenvalues appear, they are of high multiplicity. This follows by bounding below the dimension of new irreducible representations of the factor group $\Lambda / \Lambda(\RR)$. Everything we need in this direction is contained in Section \ref{algebra}. The multiplicities feature in one side of the trace formula, and the other side of the trace formula can be related to a lattice point count by choosing the right family of automorphic kernels to trace. We introduce the necessary kernels and gather some estimates on the lattice point count and spherical functions in Section \ref{analysis}. 

In the cocompact case the lattice point count plays against the multiplicity estimate via the trace formula to give a contradiction when new eigenvalues appear in a certain range. This is the approach of Sarnak and Xue. However, in the infinite volume case the trace formula does not hold as is, and must be reinterpreted as an inequality. Further repairs are needed and these were made by Gamburd in the two dimensional ($\H^2$) case to prove in \cite{GAM02}
\begin{thm}[Gamburd] Let $\Lambda = \langle A_1 , \ldots , A_k \rangle$ be a finitely generated subgroup of $\SL_2(\Z)$ with $\delta > 5/6$. Let $X(p) = \H^2 / \Lambda(p)$. For $p$ large enough
\begin{equation*}
\Omega(X(p)) \cap [\delta(1-\delta) ,  5/36 ) = \Omega(X(1)) \cap [\delta(1-\delta) , 5/36) .
\end{equation*}  
\end{thm}
The 'Collar Lemmas' in \cite{GAM02} form a key part of the generalization of Sarnak and Xue's method to infinite volume. Roughly speaking,  these state that eigenfunctions corresponding to eigenvalues $< 1/4$ in $X(p)$ are uniformly bounded through $p$ away from concentrating near infinity. The methods used to prove these do not obviously generalize to higher dimensions. The needed generalization is the thrust of this paper and appears in Section \ref{collarsection} as Lemma \ref{collar} along with the prerequisite geometry. 

The proof of Lemma \ref{collar} has a nice heuristic as follows. If one considers classical motion of a particle on a line under a step potential of height $V_0$, and the conserved energy $E$ is $< V_0$, then the particle will never enter the region covered by the step. This is due to $E = K + V$ and $K \geq 0$. In the quantum mechanical version of the same system it is no longer true that the stationary wave function is zero inside the step (quantum tunnelling). However provided $E < V_0$ is bounded away from $V_0$ we should get uniform bounds through $E$ which say the wave function cannot be arbitrarily concentrated inside the step. The positive Laplacian $\Delta_{X(\RR)}$ plays the role of a Schr\"{o}dinger operator for free dynamics on $X(\RR)$, the eigenvalues of $\Delta_{X(\RR)}$ corresponding to energy levels. We seek uniform bounds on eigenfunctions with eigenvalues bounded away from and less than $n^2/4$, which can be thought of as the escape energy, so that these eigenfunctions are bound states. Given that their energy is uniformly bounded through $\RR$ away from escape, they should not concentrate near infinity. Making formal sense of this argument constitutes the bulk of the proof.

All the machinery is brought together in Section \ref{proof} to prove the following.
\begin{thm}[Main Theorem]\label{maintheorem} Let $F$ be a totally real number field with a fixed infinite place, and $\O_F$ the ring of integers of $F$. Let $\GG = \SO(F^{n+2} , q)$ the closed $F$-subgroup of $\mathrm{GL}_{n+2}$ which preserves a quadratic form $q$ defined over $F$. Assume $\GG(\R) \cong \SO(n+1,1)$ at the fixed place and compact at the other real places. Let $\Lambda$ be a subgroup of $\GG(F) \cap \mathrm{GL}_{n+2}(\O_F)$ with the following properties.
\begin{description}
\item[Algebraic Fullness] $\Lambda$ is Zariski-dense in $\GG(\C)$ and the traces of $\Lambda$ generate $\O_F$.
\item[Geometric Regularity] The image of $\Lambda$ at the fixed place is geometrically finite, orientation preserving and torsion free as an isometry group of $\H^{n+1}$.
\item[Fractal Fullness] The Hausdorff dimension $\delta$ of the limit set of $\Lambda$ is greater than $s^0_n$, defined  
\begin{equation}
s^0_n \equiv n - \frac{2(n-1)}{(n+1)(n+2)}.
\end{equation}
\end{description}
Let $X(\RR) = \H^{n+1} / \Lambda(\RR)$ for $\RR$ an ideal in $\O_F$. Then by replacing $\Lambda$ with a finite index subgroup if necessary we have for $|\O_F / \RR|$ large enough and $\RR$ coprime to a finite set of primes
\begin{equation}
\Omega(X(\RR)) \cap \left[\delta(n-\delta) , s^0_n(n -s^0_n) \right) = \Omega(X(1)) \cap \left[\delta(n -  \delta) ,  s^0_n(n -s^0_n)\right).
\end{equation}
\end{thm}
The finite index subgroup of $\Lambda$ which we may need to pass to is the $\Lambda_1$ of Section \ref{algebra}. This is not necessary if one assumes that $\Lambda$ is contained in the kernel of the appropriate spinor norm. The index of this subgroup has independent bounds which depend on the number of generators of $\Lambda$ and the structure of a localization\footnote{This localization is given by strong approximation.} $(\O_F)_S$ respectively.
\begin{coro}[Main Corollary]\label{maincorollary}
For $\Lambda$ as before (replace $\Lambda$ with the finite index spinor norm kernel if necessary) and $\delta > s^0_n$, $\Omega(X(\RR))$ has a spectral gap. That is to say, writing $\lambda_1(X(\RR))$ for the second smallest eigenvalue of $\Delta_{X(\RR)}$, for $|\O_F / \RR|$ large and $\RR$ coprime to a finite set of primes
\begin{equation*}
\lambda_1(X(\RR)) \geq \min \left( \lambda_1(X) , s^0_n(n -s^0_n) \right).
\end{equation*}
\end{coro}
Some remarks are due before we mention applications. When $\Lambda = \G$ we obtain the result of Clozel \cite{CL03} on property $(\tau)$, albeit with a weaker bound. For any arithmetic lattice the Borel Density Theorem \cite{B60} implies Zariski-density, so our result also applies (when our geometric criteria are met). In this case we obtain the current best bound for an arbitrary arithmetic lattice in $\SO(n+1,1)$. This is due to a somewhat trivial tightening in Lemma \ref{KS} of the arguments in Kelmer and Silberman \cite{KS10} together with a nontrivial improvement in Lemma \ref{newbound}. When $\Lambda$ is of infinite index our result is entirely new. We show that our result is not vacuous in this case by constructing eligible $\Lambda$ in Section \ref{construct}.

One important application of the spectral gap is the Bourgain-Gamburd-Sarnak affine linear sieve introduced in \cite{BGS06}, \cite{BGS10}, which we recall now. 
\begin{thm}[Bourgain, Gamburd, Sarnak] Let $G \subset \mathrm{GL_n}$ be a connected, simply connected, absolutely simple algebraic group defined over $\Q$. Let $f \in \Q[G]$ be a non-zero non-unit with $t$ irreducible factors in $\Q[G]$. Let $\Lambda$ be a subgroup of $G(\Q) \cap \mathrm{GL}_n(\Z)$ finitely generated by a set $S$. We suppose the pair $(\Lambda, f)$ has the following properties.
\begin{description}
\item[Algebraic Fullness] $\Lambda$ is Zariski-dense in $G$.
\item[No local congruence obstructions]  For all integer $q \geq 2$ there exists $x \in \Lambda$ with $(f(x) , q ) = 1$.
\item[Square free Expansion] As $q$ runs through square free integers the Cayley graphs $\mathcal{H}( S \Lambda(q) , \Lambda / \Lambda(q) )$ form an expander family.
\end{description}
Then there exists $r$ such that the set of $x \in \Lambda$ such that $f(x)$ has at most $r$ prime factors is Zariski-dense in $G$. Moreover the minimal such $r$ is bounded explicitly and effectively in terms of the spectral gap in the expander family.
\end{thm}
This result is given in \cite[Theorem 1.6]{BGS10}. Using tools of additive combinatorics Bourgain and Gamburd \cite{BG08} established the expansion property for Zariski-dense $\Lambda \subset \mathrm{SL}_2(\Z)$ through prime levels. In \cite{BGS09} the expansion property in $\mathrm{SL}_2(\Z)$ for Zariski-dense $\Lambda$ is proved for square free levels, and an equivalence between expansion in Cayley graphs and the spectral gap for the spaces $\H^2 / \Lambda(q)$ is given. In case $\delta(\Lambda)$ is $\leq 1/2$ there is no discrete $L^2$ spectrum and the gap has to be interpreted as a pole free region of the meromorphically continued resolvent $(\Delta_{\H^2 / \Lambda(q)} - s(1-s))^{-1}$. Affine sieve methods are used to sharply estimate the quantity
\begin{equation*}
| \{ x \in \Lambda \: : \: |x| \leq T , \: \text{all the irreducible factors of $f$ have prime evaluation at $x$} \}|
\end{equation*}
using the non-explicit gap. Furthermore (still in \cite{BGS09}) it is shown that there is an $r$ such that 
\begin{equation} 
|\{ x \in \Lambda \: : \: |x| \leq T , \: f(x) \text{ has at most $r$ prime factors} \} | 
\end{equation}
has a good bound below, in particular implying that $f(\Lambda)$ contains infinitely many $r$-almost primes, that is, numbers which are products of at most $r$ primes. This $r$ can be determined explicitly using either the value of the $L^2$ spectral gap or the size of the pole free region of the (continued) resolvent (corresponding to $\d > 1/2$ and $\delta \leq 1/2$ respectively). 

The explicit gap is utilized in the paper of Kontorovich \cite{KON09}. There the affine sieve theory is applied (with the necessary adaptations) to the function 
\begin{equation*}
f(c,d) = c^2 + d^2
\end{equation*}
and the orbit $\O = (0,1)\G$, for $\G$ an infinite index, Zariski-dense, finitely generated subgroup of $\SL_2(\Z)$ with $\delta(\G) > 149/150$ and containing parabolics. Then using Gamburd's explicit $5/6$ gap from \cite{GAM02}, it is proved that $f(\O)$ contains infinitely many 25-almost primes. Similar methods are applied (in particular also using Gamburd's 5/6 gap) by Kontorovich and Oh \cite{KH10} to the Pythagorean orbit $\O = (3,4,5)\G$ for $\G$ a finitely generated Zariski-dense subgroup of $\SO_Q(\Z)$, 
\begin{equation*}
Q(\mathbf{x}) = x^2 + y^2 - z^2.
\end{equation*}
They consider hypotenuse $(F(\mathbf{x}) = z)$, area ($F(\mathbf{x}) =xy/12$) and product ($F(\mathbf{x}) = xyz/60$) functions. The affine linear sieve gives infinitely many $R$-almost primes in $F(\O)$ for explicit $R$ provided $\delta(\G)$ is large and there are no local congruence obstructions for the pair $(\O , F)$ ($R$ and $\delta$ depend on the function $F$ considered).

Theorem \ref{maintheorem} will yield similar applications via the affine linear sieve.

\subsection{Acknowledgments} 
I thank my advisor Alex Gamburd for sharing his work with me, and for being a source of inspiration and support. Thanks also go to Martin Weissman and Jie Qing for patient and encouraging conversations. I am grateful to Professor Sarnak for valuable feedback on a draft of this paper. The construction which features in Section \ref{construct} is due to Professor McMullen. I also thank Lior Silberman for explaining his work to me. The referee's comments have been very useful, in particular the pointer to the work of Bergeron and Clozel, and I would like to thank them for their careful reading of the manuscript.

\section{Algebra}\label{algebra}
\subsection{Notation}
Throughout this paper we use $f \ll g$ to mean that $f \leq C g$ for some constant $C$ and $f \approx g$ to mean that $g \ll f \ll g$. If subscripts are present, e.g. $\ll_\e$, this indicates that the implied constant depends on the subscripts. We view the number field $F$ and the groups $\GG$ and $\Lambda$ as fixed throughout, so all our implied constants possibly depend on these objects.

\subsection{Strong Approximation and reduction to prime powers} The aim of this section is to control the factor groups $\Lambda / \Lambda(\RR)$ and their representation theory. The group $\Lambda / \Lambda(\RR)$ is naturally a subgroup of $\G / \G(\RR)$. In \cite{GAM02} the case $n=1$ is treated at prime levels. In this case enough is known about the maximal subgroups of $\SL_2(\mathbb{F}_p)$ to provide an ad hoc proof that outside of finitely many primes $p$, $\Lambda / \Lambda(p) \cong \SL_2(\mathbb{F}_p)$. The needed bound on the nontrivial representations is that of Frobenius. 

For us the description of the factor groups will follow from the work of Weisfeiler \cite{WEI84}. This describes strong approximation for Zariski-dense subgroups of algebraic groups which amongst other things are \textit{simply connected}. We therefore need to carefully deal with the fact that $\SO$ is covered by $\Spin$. The bounds on representations at prime levels are due to Seitz and Zalesskii \cite{SZ93} and at the general level we use a result of Kelmer and Silberman \cite{KS10} together with an improvement of our own.

Recall that $F$ is a totally real number field with a fixed infinite place, and $\GG = \SO(F^{n+2} , q)$ is the closed $F$-subgroup of $\mathrm{GL}_{n+2}$ which preserves a quadratic form $q$ defined over $F$. We assume $\GG(\R) \cong \SO (n+1,1)$ at the fixed place and is compact at the other real places. The ring of integers of $F$ is $\O_F$ and $\Gamma = \mathcal{G}(F) \cap \mathrm{GL}_{n+2}(\O_F)$. We drop the dependence on $F$ and simply write $\O = \O_F$ in this section. Then we take $\Lambda$ a subgroup of $\G$ which is Zariski-dense in $\GG(\C)$ and such that the traces of $\Lambda$ in the adjoint representation generate $\O_F$. The congruence subgroups $\Gamma(\RR)$, $\Lambda(\RR)$ are kernels of the reduction map mod $\RR$ an ideal in $\O$.

Compactness of $\GG(\R_v)$ away from the fixed place implies discreteness of $\G$ at the fixed place, by using the isomorphism $\GG(\R) \cong \SO (n+1,1)$ we therefore realize $\G$ and $\Lambda$ as discrete isometry groups of $n+1$ dimensional hyperbolic space $\H^{n+1}$. We assume that $\Lambda$ is geometrically finite, hence finitely generated.

As in Weisfeiler \cite[Theorem 1.1]{WEI84}  there exists a finite set of primes $S$ such that $\GG$ can be given the structure of a group scheme over the localization $\O_S$ of $\O$ away from $S$ and $\Lambda$ is contained in $\GG_{\O_S}(\O_S)$. This results from 'clearing denominators' in the definition of $q$, we take $S$ as in  \cite[Theorem 1.1]{WEI84}. We have then an 'orthogonal $\O_S$-module' in the sense of Bass \cite{B74} by equipping $P = \O_S^{n+2}$ with $q : \O_S^{n+2} \to \O_S$, and $\GG = \SO(P , q)$ as a group scheme over $\O_S$. Following Bass \cite{B74} there is a short exact sequence of group schemes (suppressing $q$)
\begin{equation}
1 \to \mu_2 \to \mathrm{Spin} \to \GG_{\O_S} \to 1 
\end{equation}
which is exact in the \textit{fppf}\footnote{faithfully flat and finitely presented} topology on $\mathrm{Spec}(\O_S)$.

This yields the sequence in cohomology
\begin{equation}\label{eq:sequence}
1 \to \mu_2(\O_S) \to \mathrm{Spin}(\O_S) \to \GG_{\O_S}(\O_S) \to^\sigma \mathrm{H}^1(\mathrm{Spec}(\O_S) , \mu_2) .
\end{equation} 
 There is an isomorphism $H^1(\mathrm{Spec}(\O_S), \mu_2) \cong \mathrm{Discr}(\O_S)$ which converts $\sigma$ into the spinor norm $SN$. The discriminant group $\mathrm{Discr}(\O_S)$ fits into the exact sequence
 \begin{equation}
0 \to \mu_2(\O_S) \to \O_S^* \to^2 \O_S^* \to \mathrm{Discr}(\O_S) \to \mathrm{Pic}(\O_S) \to^2 \mathrm{Pic}(\O_S),
\end{equation}
which implies that $\mathrm{Discr}(\O_S)$ is a finite abelian group of exponent 2.

We let $\tilde{\Lambda}$ denote the preimage of $\Lambda$ in $\mathrm{Spin}(\O_S)$. The Strong Approximation theorem of Weisfeiler \cite{WEI84} then states that there is a finite index subgroup $\tilde{\Lambda}_0$ of $\tilde{\Lambda}$ such that the image of $\tilde{\Lambda}_0$ is dense in the group $\mathrm{Spin}(\hat{\O_S})$, where $\hat{\O_S}$ is the profinite completion of $\O_S$. In particular for $\RR$ avoiding $S$  the reduction map
\begin{equation*}
\tilde{\Lambda}_0 \to \Spin( \O_S / \RR ) \cong \Spin( \O / \RR )
\end{equation*}
is onto.

By appealing to commutativity of the diagram obtained by reducing the sequence \eqref{eq:sequence} modulo $\RR$ we get that
\begin{equation}
\Lambda / \Lambda(\RR) \supseteq \mathrm{Image}(\phi_{\RR})  = \ker \left( SN_{\RR} : \GG_{O_S}(\O / \RR) \to \mathrm{Disc}(\O / \RR) \right),
\end{equation}
where $\phi_{\RR} : \Spin(\O / \RR) \to \GG_{\O_S}(\O / \RR )$ is the covering map of finite groups and $SN_\RR$ is the spinor norm at $\RR$. If $\Lambda / \Lambda(\RR)$ is bigger than the image then there will be nontrivial representations of $\Lambda / \Lambda(\RR)$ which factor through the quotient $(\Lambda / \Lambda(\RR) ) / \mathrm{Image(\phi_\RR)}$. We let
\begin{equation}
\Lambda_1 = \ker SN\lvert_\Lambda
\end{equation}
be the kernel of the spinor norm restricted to $\Lambda$, this is a finite index normal subgroup. The index $[\Lambda : \Lambda_1]$ is bounded independently by the size of $\mathrm{Discr}(\O_S)$ and by $2^{L}$ where $L$ is the number of generators of $\Lambda$. Then we have the precise strong approximation statement
\begin{equation}
\Lambda_1 / \Lambda_1(\RR) = \ker( SN_\RR )
\end{equation}
for $\RR$ avoiding $S$.

If $\RR$ has prime factorization 
\begin{equation*}
\RR = \prod_{i=1}^{l} \P_i^{r_i}
\end{equation*}
then the group $\Lambda_1 / \Lambda_1(\RR)$ splits as a product
\begin{equation}\label{eq:factor}
\Lambda_1 / \Lambda_1(\RR) \cong \prod_{i=1}^{l} \Lambda_1 / \Lambda_1(\P_i^{r_i}),
\end{equation}
so that bounds on the size of $\Lambda_1 / \Lambda_1(\RR)$ will follow from bounds at prime power levels via $|\Lambda_1 / \Lambda_1(\RR)| = \prod_i |\Lambda_1 / \Lambda_1(\P_i^{r_i})|$. Let $\rho : \Lambda_1 / \Lambda_1(\RR) \to \Aut(V)$ be a nontrivial irreducible representation of level $\RR$, i.e. $\rho$ does not factor through a representation of $\Lambda_1 / \Lambda_1(\RR')$ for any $\RR' \lvert \RR$, $\RR' \neq \RR$. Then $\rho$ is a tensor product of irreducible representations $\rho_i :  \Lambda_1 / \Lambda_1(\P_i^{r_i}) \to \Aut(V_i)$ which are of level $\P_i^{r_i}$ respectively and 
\begin{equation}\label{eq:dimprod}
\dim \rho = \prod_{i} \dim \rho_i .
\end{equation}We have now reduced the needed argument to prime power level. We deal with the prime case first.

\subsection{Prime case}
Writing $k_\P = \O_F / \P$ for the residue field at $\P$, the previous discussion says that for $\P$ avoiding $S$ we have
\begin{equation}
\Lambda_1 / \Lambda_1(\P) =\left\{ \begin{array}{rl} 
      \Omega_{\pm}(2m, |k_\P|) & \text{if } n = 2m -2 \text{ is even} \\
      \Omega(2m+1, |k_\P|) & \text{if }  n = 2m - 1 \text{ is odd}.
    \end{array}\right.
\end{equation}

We recall some facts about these groups from \cite{SU82}. If $n = 2m - 1$ is odd the commutator subgroup $\Omega(2m+1 , |k_\P|)$ is simple and of index 2 in $\SO(2m+1,|k_\P|)$.

If $m \geq 2$ and $n = 2m-2$ even then there are two special orthogonal groups $\SO_{\pm}(2m , |k_\P| )$ and we write $\Omega_{\pm}(2m, |k_\P| )$ for the commutator subgroup. The center has size at most two and the central factor group is simple for $m \geq 3$. When $m = 2$ we have split and nonsplit versions
\begin{equation*}
\mathrm{P} \Omega_+ ( 4  , |k_\P| ) = \mathrm{PSL}_2(|k_\P|) \times \mathrm{PSL}_2(|k_\P|) ,
\end{equation*}
\begin{equation*}
\mathrm{P} \Omega_{-}( 4  , |k_\P| ) = \mathrm{PSL}_2({|k_\P|^2}) .
\end{equation*}

The following Lemma gives the needed bounds for prime levels. 
\begin{lemma}\label{dimension}
Let $\varphi$ be a nontrivial representation of $\Lambda_1 / \Lambda_1(\P)$. Then the dimension of $\varphi$ is bounded below as $|k_\P| \to \infty$ by
\begin{equation}
\dim \varphi \gg |k_\P|^{n-1} .
\end{equation}
We have for the size of the group $\Lambda_1 / \Lambda_1(\P)$
\begin{equation}\label{eq:groupsize}
| \Lambda_1 / \Lambda_1(\P) | \approx |k_\P| ^{(n+2)(n+1)/2} .
\end{equation}
\end{lemma}
\begin{proof}
By the previous discussion, outside finitely many primes we have $\Lambda_1 / \Lambda_1(\P) \cong \Omega_{(\pm)}(n+2, |k_\P| )$. The possible sizes for this group can be found in \cite{SU82}. For $n \neq 2,  4$ this is a perfect central extension of degree at most 2 of a finite Chevalley group, and lower bounds for the dimension of a nontrivial representation of such a group can be found in \cite{SZ93}. If $n = 2$ then we have $\Lambda_1 / \Lambda_1(\P) $ a degree 2 perfect central extension of $\mathrm{PSL}_2(|k_\P|) \times \mathrm{PSL}_2(|k_\P |)$ or $\mathrm{PSL}_2(|k_\P|^2 )$. At worst we have a faithful irreducible representation of $\mathrm{PSL}_2(|k_\P |)$ contained in $\varphi$. The needed bound is then well known. Finally if $n = 4$, using the accidental isomorphisms (\cite{SU82})
\begin{equation*}
\mathrm{P} \Omega_+ ( 6  , |k_\P| ) = \mathrm{PSL}_4(|k_\P|) , \quad \mathrm{P} \Omega_{-} ( 6  , |k_\P| ) = \mathrm{PSU}_4 (|k_\P|) ,
\end{equation*}
then in either case there is an associated nontrivial projective representation whose dimension can be bounded by further results tabulated in \cite{SZ93}.
\end{proof}

\subsection{Prime power case}
In this section we make use of the work of Weisfeiler \cite{WEI84} to bound the size of the group $\Lambda_1 / \Lambda_1(\P^r)$ and the work of Kelmer and Silberman \cite{KS10} along with some improvements to bound the dimension of new representations. As we have $\Lambda_1 / \Lambda_1(\P^r) = \ker SN_{\P^r}$ any new representation lifts to a nontrivial representation of the Spin group which we denote $H(\O / \P^r)$. The level structures are such that the lift is a new representation of $H( \O / \P^r)$.

For $i > 0$ let $H(\P^i)$ denote the kernel of the reduction map $H(\O / \P^r) \to H(\O / \P^{i})$, or in other words the congruence subgroup of $H(\O / \P^r)$ of level $\P^i$. Let $L$ denote the Lie algebra of $H$. We will use the following Lemma of Weisfeiler.
 
\begin{lemma}[\cite{WEI84} Lemma 5.2] \hfill \label{weilemma}
\begin{enumerate}
\item For $i > 0 $ the $H(k_\P)$ module $H(\P^i) / H(\P^{i+1})$ is isomorphic to $L(k_\P) \otimes \P^i / \P^{i+1}$, where the action on the first factor is by $\mathrm{Ad}$ and the action on the second factor is trivial.
\item The map $(x,y) \mapsto [x,y]$ maps $H(\P^i) \times H(\P^j)$ into $H(P^{i+j})$ and descends to a map
\begin{equation*}
 H(\P^i) / H(\P^{i+1}) \times  H(\P^j) / H(\P^{j+1}) \to  H(\P^{i+j}) / H(\P^{i+j+1}).
 \end{equation*}
This map is given explicitly by 
\begin{equation*} 
[x \otimes r , y \otimes s ] = [x,y] \otimes rs
\end{equation*}
when $H(\P^i) / H(\P^{i+1})$ is viewed as $L(k_\P) \otimes \P^i / \P^{i+1}$ and similarly for $j$, $j+i$.
\item If $[L(k_\P) , L(k_\P)] = L(k_\P)$ then $[H(\P) , H(\P)] = H(\P^2)$.
 \end{enumerate}
\end{lemma} 

Immediately it follows that $H(\P^i)$ is abelian for $i \geq r/2$ in light of $[H(\P^i) , H(\P^i)] \subseteq H(\P^{2i}) = \{ 1 \}$. Let $k = [r/2]$ be the integral part of $r/2$. Suppose that $\rho$ is a new representation of $H(\O / \P^r)$, i.e. the restriction $\Res^{H(\O / \P^r)}_{H(\P^{r-1})} \rho$ is not trivial. The work of Kelmer and Silberman \cite{KS10} can be paraphrased\footnote{We make a slight improvement here by noting that the proof of  \cite[Proposition 4.4]{KS10} goes through when their $e$ is $n-1$ for $n = 2,3,4$ in our indexing of $n$.}
\begin{lemma}[Kelmer, Silberman]\label{KS}
There is a character $\chi$ of $H(\P^{r-k})$ appearing in $\Res^{H(\O / \P^r)}_{H(\P^{r-k})} \rho$ which has nontrivial restriction to $H(\P^{r-1})$. Moreover for the orbit of $\chi$ under the co-Adjoint action of $H(\O / \P^r)$ we have
\begin{equation}
M \equiv | \mathrm{Orbit}_{(H(\O / \P^r) , \mathrm{co-Ad} ) } ( \chi ) | \gg \left\{ \begin{array}{rl}
      |k_\P|^{r(n-1)} & \text{if $r = 2k$ even}  \\
      |k_\P|^{(r-1)(n-1)} & \text{if $r = 2k+1$ odd}  \\
    \end{array}\right. 
\end{equation}
and there is an immediate bound below for $\dim \rho \geq M$.
\end{lemma}
For $r = 2k$ this is the result which we will use. Assume now that $r = 2k + 1$ and we will look for the natural strengthening
\begin{equation}
\dim \rho \gg |k_\P|^{r(n-1)}.
\end{equation} 
As in Lemma \ref{KS} take a character $\chi$ which appears in $\Res^{H(\O / \P^r)}_{H(\P^{r-k})} \rho$ and such that $\Res^{{H(\P^{r-k})} }_{H(\P^{r-1})} \chi$ is nontrivial. We recall some of the ingredients of the proof for our own use. The co-Adjoint action of $H(\O / \P^r)$ on the unitary dual $\widehat{H(\O / \P^{r-k})}$ descends to an action of $H( \O  / \P^k)$. Then $\widehat{H(\O / \P^{r-k})}$ is isomorphic to $L(\O / \P^k)$ via a map which intertwines the co-Adjoint and adjoint actions. Under this fixed isomorphism $\chi$ is identified with an element $X \in L(\O / \P^k)$ such that $X \neq 0 \mod \P$. An orbit-stabilizer argument then provides enough characters via the Ad-orbit of $X$. The bound on the size of the stabilizer is obtained by induction and at each stage the nonzero reduction $X_\P \in L(k_\P)$ of $X$ modulo $\P$ is all the data which is needed. 
For an explicit formulation of the connection between $X_\P$ and $\chi_0 \equiv \Res^{H(\P^{r-k})}_{H(\P^{r-1})} \chi$ let $\Tr$ denote the Galois trace $\Tr : k_\P \to \mathbb{F}_p$ where $|k_\P| = p^f$ for some $f$. Then after the identification $H(\P^{r-1}) \cong L(k_\P)$ we have for $Z \in L(k_\P)$ 
\begin{equation}
\chi(Z) = \chi_0(Z) = \exp\left(\frac{2\pi i \Tr ( B(X_\P , Z) )}{p} \right),
\end{equation}
where $B$ denotes the nondegenerate Killing form on $L(k_\P)$. 

Let $V_\chi$ be the subspace of $V$ (the vector space associated to $\rho$) upon which $H(\P^{r-k})$ acts by $\chi$. As all the $V_{\chi'}$ for $\chi'$ in the co-Adjoint orbit are isomorphic and orthogonal, if we can prove the dimension of $V_\chi$ is large we get a bound 
\begin{equation}
\dim V = \dim \rho \gg \dim V_\chi \:  | \mathrm{Orbit}_{(H(\O / \P^r) , \mathrm{co-Ad} ) } ( \chi ) |.
\end{equation}
In the next Lemma we utilize a better bound on $\dim V_\chi$ to get an improvement for $r$ odd on the dimension bound in Lemma \ref{KS} (which is using the trivial $\dim V_\chi \geq 1$).
\begin{lemma}\label{newbound}
Let $r = 2k+1 \geq 3$ and $\rho$ a new representation of $H(\O / \P^r)$. Then
\begin{equation}
\dim \rho \gg |k_\P|^{r(n-1)}
\end{equation}
with implied constant uniform through $r$ and $\P$.
\end{lemma}
\begin{proof} Take $\chi$, $V_\chi$ as before. We avoid all finitely many primes $\P$ where the characteristic $p$ of $k_\P$ is ramified in $F$. By Lemma \ref{weilemma} $H(\P^{r-k})$ is in the center of $H(\P^k) = H(\P^{r-k-1})$ and so $H(\P^k)$ preserves $V_\chi$. Let $\phi_0$ denote the subrepresentation of $\Res^{H(\O / \P^r)}_{H(\P^k)} \rho$ corresponding to $V_\chi$ and choose some irreducible representation $(\phi , W)$ of $H(\P^k)$ appearing in $(\phi_0, V_\chi)$. Our trick is to consider the  $\Hom(W,W) \cong \bar{\phi} \otimes \phi$ representation of $H(\P^k)$. As $H(\P^{r-k})$ acts as an irreducible character on $W$ it acts trivially on $\Hom(W,W)$ and so $\bar{\phi}\otimes \phi$ factors through a representation of the abelian group $H(\P^k) / H(\P^{r-k}) $ which is isomorphic to $L(k_\P)$ by Lemma \ref{weilemma}. Then $\bar{\phi} \otimes \phi$ splits as a direct sum of irreducible characters of $L(k_\P)$
\begin{equation}
 \bar{\phi} \otimes \phi \cong \bigoplus_{i = 1}^{(\dim W)^2} \theta_i.
\end{equation}
Similarly to before we can write for each $i$ and $Z \in L(k_\P)$
\begin{equation}
\theta_i (Z) =   \exp\left(\frac{2\pi i \Tr ( B(Y_i , Z) )}{p} \right)
\end{equation}
for uniquely determined $Y_i \in L(k_\P)$. Let $U = \langle Y_1, \ldots , Y_{(\dim W)^2} \rangle$ be the $\mathbb{F}_p$ vector subspace of $L(k_\P)$ spanned by the $Y_i$. In fact, via composition in $\Hom(W,W)$ any element of this space is one of the $Y_i$ as the $\theta_i$ form an abelian group. Let $U^\perp$ be the orthogonal subspace to $U$ with respect to $\Tr ( B ( \bullet, \bullet ) )$. For $v \in U^\perp$ we have  $\theta_i(v) = 1$ for all $i$ and therefore $\bar{\phi} \otimes \phi (\tilde{v}) = \mathrm{Id}_{\Hom(W,W)}$, where $\tilde{v}$ denotes a lift of $v$ in $H(\P^k)$.  The group $G_v$ generated by $\tilde{v}$ and all of $H(\P^{r-k})$ does not depend on the lift and is an abelian group by Lemma \ref{weilemma}. Therefore $\Res^{H(\P^k)}_{G_v} \phi$ splits as a direct sum of irreducible characters. As $\bar{\phi} \otimes \phi (\tilde{v})$ is trivial all of these characters must be the same and therefore $\phi(\tilde{v})$ is a scalar multiple of the identity on $W$. For any $\tilde{z} \in H(\P^k)$ reducing mod $H(\P^{r-k})$ to $z \in L(k_\P)$ we have for the image of the group commutator
\begin{equation}
1 = \phi( [\tilde{v},\tilde{z}]_{\mathbf{Grp}} )  = \phi( [v,z]_{\mathbf{Lie}} ),
\end{equation}
where $[v,z]_{\mathbf{Lie}}$ is viewed as an element of $H(\P^{r-1})$ by Lemma \ref{weilemma}. Moreover
\begin{equation*}
 \phi( [\tilde{v},\tilde{z}]) = \chi_0( [v,z] ) = \exp \left( \frac{2\pi i \Tr ( B(X_\P , [v,z]) )}{p} \right)
 \end{equation*}
 in the notation from our previous discussion so in particular 
 \begin{equation}
 \Tr ( B(X_\P , [v,z])) = 0
 \end{equation}
 for all $z \in L(k_\P)$, or what is the same by ad-invariance of the Killing form
 \begin{equation}
 \Tr( B([v,X_\P] , z )) = 0.
 \end{equation}
 By nondegeneracy it follows that $[v, X_\P] = 0$, i.e. $v$ is in the centralizer $C_{L(k_\P)}(X_\P)$. As in \cite{KS10} we have\footnote{This inequality is the key to the inductive step of the proof of Lemma \ref{KS}, so it is fitting that it is also the ingredient here.}
 \begin{equation}\label{eq:centralizer}
 \dim_{k_\P} C_{L(k_\P)}(X_\P) \leq \dim_{k_\P}L(k_\P) - 2(n-1)
 \end{equation}
 and as we have shown $U^\perp \subseteq C_{L(k_\P)}(X_\P)$ it follows that 
\begin{align*}
\dim_{\mathbb{F}_p} U^{\perp} \leq \dim_{\mathbb{F}_p} C_{L(k_\P)}(X_\P) &= f \dim_{k_\P} C_{L(k_\P)}(X_\P) \\
&\leq f (\dim_{k_\P}L(k_\P) - 2(n-1)) \\
&= \dim_{\mathbb{F}_p} L(k_\P) - 2 f (n-1)
\end{align*}
where $f$ is the inertia degree. Then immediately $\dim_{\mathbb{F}_p} U \geq 2f ( n - 1 )$ so that
\begin{equation}\label{eq:Wbound}
\dim W \geq \sqrt{ \# \text{ of $Y_i $} } = \sqrt{ |\mathbb{F}_p|^{\dim_{\mathbb{F}_p} U } } \geq \sqrt{ |\mathbb{F}_p|^{2 f ( n - 1 ) }}=  |k_\P|^{n-1}.
\end{equation}
Finally $W$ is a subspace of $V_\chi$ and as we remarked before we have the bound
\begin{equation}
\dim \rho \geq \dim V_\chi | \mathrm{Orbit}_{(H(\O / \P^r) , \mathrm{co-Ad} ) } ( \chi ) | \geq \dim W | \mathrm{Orbit}_{(H(\O / \P^r) , \mathrm{co-Ad} ) } ( \chi ) | .
\end{equation}
Inserting the bound from Lemma \ref{KS} on the size of the orbit together with the bound in \eqref{eq:Wbound} gives us the desired result
\begin{equation}
\dim \rho \gg |k_\P|^{n-1} |k_\P|^{(r-1)(n-1)} = |k_\P|^{r(n-1)} .
\end{equation}
\end{proof}
We remark that the idea of this proof should hold in the more general setting of Lemma \ref{weilemma}, provided centralizer bounds analogous to \eqref{eq:centralizer} can be obtained.

\subsection{General result}
Now we have all that is required for the main result of this section.
\begin{lemma}\label{generalfinite} If \begin{equation}
\RR = \prod_{i=1}^{l} \P_i^{r_i}
\end{equation}
is the prime factorization of $\RR$ an ideal in $\O_F$, then for the size of the group $\Lambda_1 / \Lambda_1(\RR)$ we have
\begin{equation}
|\Lambda_1 / \Lambda_1(\RR)| \approx \prod_{i = 1}^{l}  |k_{\P_i}| ^{r_i(n+1)(n+2)/2} = |\O_F / \RR |^{(n+1)(n+2)/2}.
\end{equation}
Any representation $\rho$ of $\Lambda_1 / \Lambda_1(\RR)$ of level $\RR$ has dimension
\begin{equation}
\dim \rho \gg \prod_{i = 1}^{l} |k_{\P_i}|^{r_i(n-1)} = |\O_F / \RR |^{n-1}.
\end{equation}
\end{lemma}
\begin{proof}
For the size of the group it is sufficient to give the size of each of the $\Lambda_1 / \Lambda_1(\P_i^{r_i} )$ by \eqref{eq:factor}. This is obtained from the bound at prime level in Lemma \ref{dimension} together with Lemma \ref{weilemma} which gives the size of the factor $H(\P_i^j) / H(\P_i^{j+1})$ as $ |k_{\P_i} |^{(n+2)(n+1)/2} = |L(k_{\P_i})|$. The bound on the dimension of $\rho$ follows from the discussion leading up to \eqref{eq:dimprod} together with the bounds obtained at prime level in Lemma \ref{dimension}, at even power of prime level in Lemma \ref{KS} and at odd power of prime level in Lemma \ref{newbound}. 
\end{proof}

\section{Eigenfunction Estimates}\label{collarsection}
\subsection{The geometry near infinity}
Recall that we are considering geometrically finite, torsion free and orientation preserving subgroups $\Lambda , \Lambda(\RR)$ of $\Isom^{+}(\H^{n+1})$ (here we replace $\Lambda$ with the $\Lambda_1$ of the previous section if necessary). The geometry of the quotients $X(\RR) = \H^{n+1} / \Lambda(\RR)$ is well understood, although less is usually said about the nature of the covering maps $\pi_\RR : X(\RR) \to X$. We aim in this section to describe the geometry of the covering maps near infinity. A good description of the geometry at a fixed level can be found in the paper of Mazzeo and Phillips \cite{MP90}. We follow the notation of \cite{GUI09}, where Guillarmou considers slightly less general spaces (with finite holonomy in the cusps) but provides very useful analytic lemmas which can be extended to our case without difficulty.

The space $\H^{n+1}$ has a natural compactification by adding a sphere at infinity $S^n_\infty$. The action of $\Lambda$ on $\H^{n+1}$ extends to the sphere. If $o$ is a point in $\H^{n+1}$ then the orbit $\Lambda o$ accumulates on a subset of the boundary denoted $L(\Lambda)$, the limit set of $\Lambda$ (that this set is independent of $o$ is a general property of nonpositive curvature). The complement $S^n_\infty - L(\Lambda)$ is called the domain of discontinuity. It follows that $\Lambda$ acts discretely and properly discontinuously on the domain of discontinuity. As $\Lambda$ is geometrically finite it has a Dirichlet fundamental domain $\F$ in $\H^{n+1}$ which is finitely faced by totally geodesic hypersurfaces. If $\Lambda$ is infinite index in $\G$ then this fundamental domain necessarily extends to the boundary, and the bounding hypersurfaces meet the boundary in subsets of $n-1$ dimensional spheres. Similar statements hold for $\Lambda(\RR)$, in particular a fundamental domain $\F_\RR$ for $\Lambda(\RR)$ is paved by images of $\F$ under coset representatives of $\Lambda / \Lambda(\RR)$. 

The elements of $\G$ can be classified by their fixed points on the boundary. Either $\g$ fixes two points on the boundary and the geodesic between these two points, in which case it is called \textit{hyperbolic}, or $\g$ has one fixed point and acts by Euclidean motions on horospheres tangent to this point. In the latter case $\g$ is called \textit{parabolic}.

Away from parabolic fixed points, the region where $\F$ meets the boundary can be covered by finitely many charts isometric to regions
\begin{align}
M_r &= \{(x,y) \in (0,\infty) \times \R^n \: ; \: x^2 + |y|^2 < 1 \} , \\  
g_r &= x^{-2}(dx^2 + dy^2), \nonumber
\end{align}
which we call regular neighbourhoods. These can be chosen sufficiently small so that they project isometrically to the quotient $X$. The pulled back charts cover a corresponding region in $X(\RR)$. The remaining neighbourhoods are near parabolic fixed points. Suppose that $\infty$ in the upper half space model is the fixed point of a parabolic element, by conjugating if necessary. Let $\Lambda_\infty$ consist of all elements of $\Lambda$ with this fixed point. This group is purely parabolic by discreteness and can be thought of as acting as Euclidean isometries on any horosphere, which is isomorphic to $\R^n$. Some of the facts which follow come from the theory of Bieberbach groups for which the reader can consult the notes of Thurston \cite{TH97}.  

Following \cite{MP90}, the group $\Lambda_\infty$ contains a maximal normal free abelian subgroup $\Lambda_a$ of finite index. We define the rank $k$ of the cusp at $\infty$ to be the rank of $\Lambda_a$. There exists a maximal affine subspace $\R^k \subset \R^n$ fixed by $\Lambda_\infty$. The subgroup $\Lambda_a$ acts on this space by translations and the quotient $F_k \equiv \R^n / \Lambda_\infty$ is the total space of a flat vector bundle of rank $n - k$ over a compact flat base manifold $B_k$. This $B_k$ can realized concretely as $\R^k / \G_\infty$. It is covered by a flat $k$-torus $T^k = \R^k / \Lambda_a$ by the usual Galois correspondence, the covering map coming from the map of flat vector bundles $\tilde{F}_k \equiv \R^n / \Lambda_a \to \R^n / \Lambda_\infty$ restricted to the zero section. We use $y$ for a local coordinate in the fibre coming from $\R^{n-k}$, and $z$ for a coordinate on $B_k$ coming from the covering $\R^k \to \R^k / \Lambda_\infty$. To cover the regions at infinity in $X$ which come from parabolic fixed points we can use charts isometric to rank $k$ cuspidal neighbourhoods
\begin{align}
M_k &= \{(x,[y,z]) \in (0,\infty) \times F_k \: ; \: x^2 + |y|^2 > 1 \} , \\
g_k &= x^{-2}(dx^2 + dy^2 + dz^2), \nonumber
\end{align}
where we are writing $[y,z]$ for local trivializing coordinates. The quadratic differential $dz^2$ refers to a flat metric on $B_k$. Note also that $|y|^2$ is a well defined function, as changing trivialization affects an orthogonal transformation on $y$ and similarly $dy^2$ is defined independently of trivialization. In general there are finitely many cusps of each rank but to simplify the discussion we assume that there is only one neighbourhood of each cuspidal type $M_k$. We drop the isometries which identify the sets in $X$ with the model neighbourhoods and think of the $M_r$, $M_k$ as sets in $X$. It can be arranged so that all the cuspidal neighbourhoods are disjoint.

We consider now the covering maps $\pi_\RR : X_\RR \to X$. If $q$ is a parabolic fixed point of $\Lambda$ with stabilizer $\Lambda_q$ then $\Lambda_q \cap \Gamma(\RR)$ also fixes this point and is the stabilizer in $\Lambda(\RR)$,  i.e. $\Lambda(\RR)_q = \Lambda_q \cap \Gamma(\RR)$. In addition, $\Lambda_a \cap \Gamma(\RR)$ is the maximal normal free abelian subgroup $\Lambda(\RR)_a$ in $\Lambda(\RR)_q$. There is then a map of flat rank $n-k$ vector bundles $F_{k,\RR} \equiv \R^n / \Lambda(\RR)_q \to F_k$ which when restricted to the zero section gives a covering map of flat compact manifolds $B_{k,\RR} \equiv \R^k / \Lambda(\RR)_q \to B_k$. If one considers the images of $\F$ under coset representatives of $\Lambda / \Lambda(\RR)$ one sees that $\pi_\RR^{-1}(M_k)$ is isometric to a disjoint union of isometric cuspidal neighbourhoods
\begin{equation}
\pi_\RR^{-1}(M_k) \cong \coprod_{i = 1}^{m_{k,\RR}} M_{k,\RR} 
\end{equation}
where
\begin{align}
M_{k,\RR} &= \{(x,[y,z]) \in (0,\infty) \times F_{k,\RR} \: ; \: x^2 + |y|^2 > 1 \} , \\
g_k &= x^{-2}(dx^2 + dy^2 + dz^2), \nonumber
\end{align}
and as before, $y$ and $z$ refer to trivializing coordinates for the flat vector bundle.  If $l_{k,\RR}$ is the degree of the covering of base manifolds $B_{k,\RR} \to B_k$ then we have $m_{k,\RR} l_{k,\RR} = |\Lambda / \Lambda(\RR)|$, the degree of the covering map $\pi_\RR$. As we will work with the covering tori $T^k_\RR \equiv \R^k / \Lambda(\RR)_a \to T_k$ it is salient to note that the obvious diagram of covering maps involving the tori $T^k_\RR$, $T^k$ and base manifolds $B_{k,\RR}$, $B_\RR$ is commutative. The covering map at $M_k$ can be given explicitly with respect to these neighbourhoods, indeed it is directly induced by the map $F_{k,\RR} \to F_k$ so that the $x$ coordinate is preserved. As for the regular neighbourhoods, it can be arranged so that the preimage $\pi_\RR^{-1}(M_r)$ of each individual $M_r$ is isometric to $|\Lambda / \Lambda(\RR)|$ disjoint copies of $M_r$. Then the covering map in each disjoint copy is given by the identity with respect to these charts.

Now we seek compactification coordinates for these charts. The transformation defined locally and which only affects the coordinate in $\R^+$ and in the fibre
\begin{equation}
(x,[y,z]) \mapsto (t , [u,z] ) = \left ( \frac{x}{x^2 + |y|^2} , [\frac{-y}{x^2 + |y|^2} , z] \right)
\end{equation}
actually (by flatness) gives a diffeomorphism from $(M_k , g_k)$ to 
\begin{equation*}
\{ (t,[u,z]) \in (0,\infty) \times F_{k} \: ; \: t^2 + |u|^2 < 1 \}.
\end{equation*}
The pushed forward metric is
\begin{equation}
t^{-2}(dt^2 + du^2 + (t^2 + |u|^2)^2 dz^2 ).
\end{equation}
These coordinates allow the charts $M_k$ to be smoothly compactified by adding a $\{ t = 0\}$ portion to form $\bar{M_k}$. The charts $M_r$ also naturally compactify by adding a $\{ x = 0 \}$ part and all these boundary pieces join together to give a smooth boundary $\delta \bar{X}$ to $X$. The compactification of $X$ is denoted $\bar{X}$, this has the structure of a smooth compact manifold with boundary $\delta\bar{X}$. A similar procedure takes place to smoothly compactify $X(\RR)$ to $\bar{X}(\RR)$. For each $M_k$ there is a cusp submanifold isomorphic to $B_k$ of $\delta\bar{X}$ corresponding to the zero section of $F_k$ at $t = 0$. The cusp submanifold of $\delta\bar{X}$ coming from $M_k$ is denoted $b_k$ and we define
\begin{equation}
B \equiv \delta\bar{X} - \coprod _{k=1}^n b_k .
\end{equation}
Following the same procedure at level $\RR$ we get a regular boundary part $B(\RR)$ which naturally covers $B$. A collar neighbourhood of the boundary could now be expected by Milnor's Collar Neighbourhood Lemma \cite{MIL65}; the following proposition gives us some fine control over the geometry in such a neighbourhood.
\begin{prop}\label{collarmetric} In some collar neighbourhood $(0,\e)_\rho \times \delta \bar{X}$ of $\delta \bar{X}$ the hyperbolic metric is given by
\begin{equation}
g = \frac{d \rho^2 + h(\rho) }{\rho^2}
\end{equation}
for some smooth family of symmetric tensors $h(\rho)$ on $\delta \bar{X}$, depending smoothly on $\rho$, positive for $\rho > 0$ with $h(0) = h_0$ positive on $B$ and satisfying
\begin{equation*}
h(\rho) = du^2 + (\rho^2 + |u|^2)^2 dz^2
\end{equation*}
in each $\bar{M}_k$. Moreover $\rho = t$ in $\bar{M}_k$.
\end{prop}
This appears as a discussion in \cite{GUI09}. It extends to include maximal rank cusps by interpreting the $u$ variable as absent. Also, Guillarmou is only considering the case of rational cusps, but the proof goes through in our slightly more general case as it only relies on a PDE being non-characteristic away from the cusps.

The boundary defining function $\rho$ lifts from $X$ to $X(\RR)$. By our previous remarks on the nature of the covering maps in local charts, by making $\e$ small enough so $\rho^{-1}((0,\e))$ does not escape any of the local charts we can use the same $\e$ for all $\RR$. The metric provided also lifts to the collar neighbourhood $(0,\e)_{\pi_\RR^* \rho} \times \delta\bar{X}(\RR)$. By examining the covering map in the cuspidal regions we get an exact form for the metric at each level $\RR$. 

\subsection{Bounds below in the cusps} Here we examine the cuspidal regions near infinity. There the metric is exact and we can get an exact result. In the local $[u,z]$ coordinates corresponding to the boundary part in $\bar{M}_k$, define the product regions $N_k(R) = (0, 1/R) \times \{ |u|^2 < 3/4 \}$, and $N_{k,\RR}(R) = \pi_{\RR}^{-1}( N_k(R) )$. In the maximal rank cusp $N_n(R) = (0, 1/R) \times B_n$. Note that by choosing $R_0$ large enough, we can assume that for all $R > R_0$, $N_k(R) \subset \bar{M}_k$ (hence the corresponding result at level $\RR$), and we make this increase for $R_0$ immediately if necessary, and consider $R>R_0$.

We separate into two cases depending on whether the holonomy representation
\begin{equation}
h : \pi_1(T^k) \to \mathrm{O}(n-k)
\end{equation}
has finite image. If it does then we can pass to a finite locally isometric covering  $f : (\overline{T^k}, f^* dz^2) \to (B_k, dz^2) $ such that $f^* F_{k}$ is trivial. The torus $\overline{T^k}$ is the quotient $\R^k / \ker(h)$. Moreover the holonomy representation of $T^k_\RR$ is the restriction to a smaller group so remains finite at all levels. The reader can also note that the finite holonomy property is always satisfied for $n=2$. We cover finite holonomy in the following Lemma. 
\begin{lemma}\label{cusppos1}  If the holonomy representation of $T^k$ has finite image then
\begin{equation}
\Delta_{X(\RR)}\lvert_{C^{\infty}_0 ( N_{k,\RR}(R) )} \geq n^2 / 4.
\end{equation}
\end{lemma}
 
\begin{proof} First we make a lifting argument to simplify the case. Let $\phi \in C^{\infty}_0 ( N_{k,\RR}(R) )$.  We lift $\phi$ to a region covering $N_{k,\RR}(R)$
\begin{equation*}
\overline{N_{k,\RR}(R)} \equiv \{ (t,u,z) \in (0,1/R) \times \R^{n-k} \times \overline{T^k_\RR} \: : \: |u|^2 < 3/4 \}
\end{equation*}
equipped with the pulled back metric via $1 \times f$. Then note
\begin{align*}
\langle \Delta_{\overline{N_{k,\RR}(R)}} (1 \times f)^* \phi , (1 \times f)^*  \phi \rangle_{\overline{N_{k,\RR}(R)}} &= | h(\pi_1(T^k_\RR)) | \times \langle \Delta_{N_{k,\RR}(R)}\phi , \phi \rangle_{N_{k,\RR}(R)} ,\\
 \| (1 \times f)^*  \phi \|^2_{\overline{N_{k,\RR}(R)}}  &= | h(\pi_1(T^k_\RR)) | \times  \| \phi \|^2_{N_{k,\RR}(R)}  ,
\end{align*}
and $(1 \times f)^* \phi \in C^{\infty}_0(\overline{N_{k,\RR}(R)})$ so it is sufficient to prove the result when the cross section is a trivial bundle over a torus. Then assume that $F_{k,\RR} \cong \R^{n-k} \times T^k_\RR$, so that
\begin{equation*}
N_{k,\RR}(R) = \{ (t,u,z) \in (0,1/R) \times \R^{n-k} \times \overline{T^k_\RR} \: : \: |u|^2 < 3/4 \}.
\end{equation*}
Now we introduce the change of coordinates $t = e^{-\tau}$, and then conjugate the Laplacian by the function $|g|^{1/4} = e^{n\tau / 2}(e^{-2\tau} + |u|^2)^{k/2}$. Clearly $\phi \mapsto |g|^{1/4}\phi$ preserves ${C^{\infty}_0 ( N_{k,\RR}(R) )}$. The conjugated Laplacian acts on the $L^2$ space defined by the  volume element $d\tau \wedge du \wedge dz$. One has for $L = |g|^{1/4} \Delta |g|^{-1/4}$ in coordinates $(\tau,u,z)$
\begin{equation}\label{eq:LCUSP}
L = - \partial_\tau^2 + e^{-2\tau}\Delta^{\mathrm{Euclidean}}_u + \frac{e^{-2\tau}}{(e^{-2\tau} + |u|^2)^2}\Delta^{\mathrm{Flat Toroidal}}_z + \frac{n^2}{4}.
\end{equation} The main part of this calculation appears in \cite[eq. 5.2]{GUI09}. This completes the proof as the first 3 terms can be easily verified nonnegative on $C^{\infty}_0 ( N_{k,\RR}(R) )$ with the product measure.
\end{proof}
As remarked before the reader can skip the next Lemma if they are interested only in the case $n=2$. For the remaining cases the idea is that functions on a flat bundle with nondiscrete holonomy are equivalent to functions on a bundle with discrete holonomy where the previous Lemma can be applied. Moreover we can find finite holonomy bundles 'arbitrarily close' to the original, and the action of the Laplacian is continuous in some sense with respect to this approximation. This perturbation argument is due to Mazzeo and Phillips as they use it in \cite[Lemma 5.12]{MP90}.
\begin{lemma}\label{cusppos} 
In any cuspidal neighbourhood $N_{k,\RR}(R)$,
\begin{equation}
\Delta_{X(\RR)}\lvert_{C^{\infty}_0 ( N_{k,\RR}(R) )} \geq n^2 / 4.
\end{equation}
\end{lemma}
\begin{proof}
We assume now that the holonomy representation has infinite image, or we are done by Lemma \ref{cusppos1}. We drop the $\RR$ dependence which does not matter as we only deal with one level at a time. By the same lifting argument as before it is sufficient to consider the case when $B_{k} = T^k$ so the cross section is a flat bundle $F_{k}$ over a torus. Then conjugating the parabolic fixed point to $\infty$ in the upper half space model $\{ ( x, u, z ) \in \R^+ \times \R^{n-k} \times \R^{k} \}$ we have $\Lambda_\infty = \Lambda_a = \pi_1(T^k)$ a free abelian group of rank $k$ and preserving the Euclidean horosphere $\{ x = 1 \}$ which we take as a cover. The image of the holonomy representation is commutative so each element in the image has the same invariant subspaces. For simplicity we proceed as though $n = 3$ and $k = 1$, so that $\pi_1(T^1)$ is infinite cyclic with generator $\g$. By hypothesis (and further conjugating if necessary) $\g$ acts on the plane $\{ ( x, u, z ) \in \R^+ \times \R^2 \times \R : x = 1 \}$ by
\begin{equation*}
\g : (u, z) \mapsto  \left( \left( \begin{array}{rl}
      \cos \theta & \sin\theta\\
      -\sin \theta & \cos \theta
    \end{array}\right) u , z + 1 \right)
\end{equation*}
with $\theta$ irrational. We consider the perturbing map 
\begin{equation*}
p_\eta : (u ,z)  \mapsto \left( \left( \begin{array}{rl}
      \cos (\eta z) & \sin(\eta z)\\
      -\sin (\eta z) & \cos (\eta z)
    \end{array}\right) u , z \right)
\end{equation*}
where $\eta + \theta$ is rational and $\eta$ will be chosen arbitrarily small in what follows. We compute
\begin{equation*}
p_\eta \g p_\eta^{-1} (u,z) = \left( \left( \begin{array}{rl}
      \cos (\theta + \eta ) & \sin(\theta + \eta )\\
      -\sin (\theta + \eta ) & \cos (\theta + \eta )
    \end{array}\right) u , z + 1 \right).
\end{equation*}
This implies $p_\eta$ descends to a bundle isomorphism from $F_{k} \to F_k^\eta$ where $F_k^\eta$ is the total space of a bundle over $T^1$ with rational holonomy. It also induces a map (also denoted $p_\eta$) to a new cuspidal region $N^\eta_k(R)$ in the obvious way. Now take $\phi \in  C^{\infty}_0 ( N_{k}(R) )$. It is not hard to show\footnote{Something very similar appears in the proof of \cite[Lemma 5.12]{MP90}.} that
\begin{equation}\label{eq:conv}
\langle \Delta_{N^\eta_k(R)} p_{\eta *} \phi , p_{\eta *} \phi \rangle_{L^2(N_k^\eta(R))} = \langle p^*_\eta \Delta_{N^\eta_k(R)} p_{\eta *} \phi , \phi \rangle_{L^2(N_k(R))} \to \langle  \Delta_{N_k(R)}  \phi , \phi \rangle_{L^2(N_k(R))}
\end{equation}
as $\eta \to 0$. We have also
\begin{equation*}
\langle \Delta_{N^\eta_k(R)} p_{\eta *} \phi , p_{\eta *} \phi \rangle_{L^2(N_k^\eta(R))} \geq n^2/4 \|  p_{\eta *} \phi \|^2_{L^2(N_k^\eta)(R)} = n^2/4  \|  \phi \|^2_{L^2(N_k(R))}
\end{equation*} 
by Lemma \ref{cusppos1}. Then choosing a sequence of $\eta_i \to 0$ with $\theta + \eta_i$ rational in equation \eqref{eq:conv} we have
\begin{equation*} 
\langle  \Delta_{N_k(R)}  \phi , \phi \rangle_{L^2(N_k(R))} \geq n^2/4  \|  \phi \|^2_{L^2(N_k(R))}
\end{equation*}
which is the required bound below. The same idea works in the general case with more 'perturbation directions'.
\end{proof} 

\subsection{Bounds below in regular neighbourhoods}
Let $B_0 \subset \delta \bar{X}$ be the compact manifold with boundary 
\begin{align}
B_0 &= \delta \bar{X} - \bigcup_{k = 1}^{n-1} \{ (u,z) \in \delta \bar{M}_k : |u|^2 < 1/4 \}, \\
 \delta &B_0 = \bigcup_{k = 1}^{n-1} \{ (u,z) \in \delta \bar{M}_k : |u|^2 = 1/4 \} .\nonumber
\end{align}
By Proposition \ref{collarmetric}, after the change of coordinates $\rho = e^{-\tau}$ the hyperbolic metric on $\pi_{\RR}^{-1}((0,1/R) \times B_0)$ is of the form
\begin{equation*}
g = d\tau^2 + e^{2\tau }\pi_\RR^{*} \g(\tau) .
\end{equation*}
Here $\g(\tau) = h(e^{-\tau})$ with $h$ as in Proposition \ref{collarmetric}. The Laplacian on $\pi_{\RR}^{-1}((\ln R , \infty)_\tau \times B_0)$  takes the form (as in \cite[eq. 5.1]{GUI09})
\begin{equation}\label{eq:laplacianform}
\Delta_{X(\RR)} = - \partial_\tau^2 - n\partial_\tau - \frac{1}{2}\pi_{\RR}^*\left(\mathrm{Tr}(\g^{-1}(\tau).\partial_\tau \g )\right)\partial_\tau + e^{-2\tau}\Delta_{\pi_{\RR}^*\g(\tau)} .
\end{equation}
Our next Lemma prepares the way by giving the Laplacian on $\pi_{\RR}^{-1}((\ln R , \infty)_\tau \times B_0)$ as a 'polynomial' in $\partial_\tau$ and $ \nabla_{\pi_{\RR}^*\g(0)}$.
\begin{lemma}
The Laplacian on  $\pi_{\RR}^{-1}((\ln R , \infty)_\tau \times B_0)$ can be written
\begin{align}
\Delta_{X(\RR)} &= - \partial_\tau^2 - n\partial_\tau + e^{-2\tau}\Delta_{\pi_{\RR}^* \g(0)}  \label{eq:laplacianunconjugated} \\
& e^{-\tau} (\pi_{\RR}^{*} f) \partial_\tau + e^{-3\tau} \left( \pi_{\RR}^* a . \nabla_{\pi_{\RR}^*\g(0)} + \mathrm{div}_{\pi_{\RR}^*\g(0)} . \pi_{\RR}^*b  . \nabla_{\pi_{\RR}^*\g(0)} \right), \nonumber
\end{align}
where $f$ (resp. $a$, resp. $b$) is a smooth bounded function (resp. one form, resp. endomorphism of the tangent bundle) on $(\ln R , \infty)_\tau \times B_0$.
\end{lemma}
\begin{proof}
 By the smoothness of the family of metrics $h(\rho)$, we can write
\begin{align}
\g^{-1}(\tau) &= (1 + e^{-\tau} b )\g^{-1}(0), \label{al2}\\
 | \g(\tau) |^{1/2}  &=(1 + e^{-\tau} c )| \g(0) |^{1/2}, \label{al3}\\
 | \g(\tau) |^{-1/2} &=  (1 + e^{-\tau} \tilde{c} ) | \g(0) |^{-1/2}  ,  \label{al4} 
\end{align}
with (by multiplying \eqref{al3} and \eqref{al4})
\begin{equation}
c+ \tilde{c} + e^{-\tau}c\tilde{c} \equiv 0. \label{eq:identity}
\end{equation}
The quantities $b$, $c$ and $\tilde{c}$ are smooth bounded 2-tensors and functions respectively on $(\ln R , \infty)_\tau \times B_0$. It follows from the compactness of $B_0$ that any fixed finite number of derivatives of $b$, $c$ and $\tilde{c}$ are smooth and bounded on $(\ln R , \infty)_\tau \times B_0$. The analogous statements hold for $\pi_{\RR}^* \g(\tau)$ by replacing $b$, $c$ and $\tilde{c}$ with their lifts.

We calculate that, using \eqref{al3} and \eqref{al4} and writing $d$ for the exterior derivative on $B_0$,
\begin{align}
\mathrm{div}_{\g(\tau)} &= \mathrm{div}_{\g(0)} + e^{-\tau}(1 + e^{-\tau} \tilde{c}) d (c ) \nonumber \\
&+ e^{-\tau}( c+ \tilde{c}+ e^{-\tau} c\tilde{c}) \mathrm{div}_{\g(0)}\nonumber \\ 
&= \mathrm{div}_{\g(0)} + e^{-\tau} \omega  , \label{eq:div}
\end{align}
where the term on the second line vanished due to \eqref{eq:identity} and $\omega$ is a smooth bounded one form on $(\ln R , \infty)_\tau \times B_0$ (we use boundedness of derivatives of $c$ here). By \eqref{al2},
\begin{equation}\label{eq:grad}
\nabla_{\g(\tau)} = \nabla_{\g(0)} + e^{-\tau} b . \nabla_{\g(0)}.
\end{equation}
Now using that
\begin{equation}
\Delta_{\g(\tau)} = \mathrm{div}_{\g(\tau)} . \nabla_{\g(\tau)}
\end{equation}
our previous formulae \eqref{eq:div} and \eqref{eq:grad} give
\begin{align}
\Delta_{\g(\tau)} &=  \left (\mathrm{div}_{\g(0)} + e^{-\tau} \omega  \right). \left(  \nabla_{\g(0)} + e^{-\tau} b . \nabla_{\g(0)} \right) \\
&=\Delta_{\g(0)} + e^{-\tau} \omega . \left(1 + e^{-\tau}b \right) . \nabla_{\g(0)} + e^{-\tau}  \mathrm{div}_{\g(0)}.  b . \nabla_{\g(0)} \\
&= \Delta_{\g(0)} + e^{-\tau} \left( a . \nabla_{\g(0)} +  \mathrm{div}_{\g(0)}  . b   .\nabla_{\g(0)} \right) ,
\end{align}
where $a$ (resp. $b$) is a smooth bounded one form (resp. endomorphism of the tangent bundle) on $(\ln R , \infty)_\tau \times B_0$. To get the analogous result at level $\RR$ we can repeat the argument and note that all the quantities which appear are the lifts of their counterparts in the previous discussion. This yields
\begin{equation}\label{eq:crosssectionperturbation}
\Delta_{\pi_{\RR}^* \g(\tau)} =\Delta_{\pi_{\RR}^* \g(0)} +  e^{-\tau} \left( \pi_{\RR}^* a . \nabla_{\pi_{\RR}^*\g(0)} + \mathrm{div}_{\pi_{\RR}^*\g(0)} . \pi_{\RR}^*b  . \nabla_{\pi_{\RR}^*\g(0)} \right) .
\end{equation}

Similar arguments show that
\begin{equation}\label{eq:traceterm}
-\frac{1}{2}\pi_{\RR}^*\left(\mathrm{Tr}(\g^{-1}(\tau).\partial_\tau \g )\right) = e^{-\tau} \pi_{\RR}^{*} f
\end{equation}
where $f$ is a smooth bounded function. Substituting \eqref{eq:crosssectionperturbation} and \eqref{eq:traceterm} into \eqref{eq:laplacianform} gives the desired expression \eqref{eq:laplacianunconjugated}.
\end{proof}

To proceed, we conjugate the Laplacian by the function 
\begin{equation*}
G_{\RR} \equiv \pi_{\RR}^*(|g|/|\g(0)|)^{1/4}
\end{equation*}
 so as to act on the product space associated to the volume element $d \tau \wedge \mu_{\pi_{\RR}^* \g(0)}$. The conjugated Laplacian 
\begin{align*}
L(\RR;R) &: C^{\infty}_0(\pi_\RR^{-1} ((\ln R , \infty) \times B_0)) \to C^{\infty}_0(\pi_\RR^{-1} ((\ln R , \infty) \times B_0)) \\
L(\RR;R) &\equiv G_\RR \Delta_{X(\RR)} G_\RR^{-1} 
\end{align*}
takes the form
\begin{equation}
L(\RR;R) = L_0(\RR; R) + E(\RR; R)
\end{equation}
where 
\begin{equation}\label{eq:L0FORM}
L_0(\RR ; R) = -\partial_\tau^2 + e^{-2\tau} \Delta_{\pi_\RR^*\g(0)} + n^2/4 .
\end{equation}
The operator $\Delta_{\g(0)} $ refers to the Laplacian on $(B_0,\g(0))$ with Dirichlet boundary conditions. In the next Lemma we compute the error term $E(\RR; R)$.
\begin{lemma}\label{errorform}In the region  $\pi_\RR^{-1} ((\ln R , \infty)_\tau \times B_0)$ the error term $ E(\RR;R)$ is given by
\begin{equation}
 E(\RR;R)= e^{-\tau} A_\RR \partial_\tau + e^{-3\tau} B_\RR . \nabla_{\pi_\RR^* \g(0)}  + e^{-3\tau} \mathrm{div}_{\pi^*_\RR \g(0) } . C_\RR . \nabla_{\pi_\RR^* \g(0)} + e^{-\tau} D_\RR
\end{equation}
where $A_\RR$ and $D_\RR$ are $\RR$-uniformly bounded functions, $B_\RR$ is an $\RR$-uniformly bounded one form, and $C_\RR$ is an $\RR$-uniformly bounded endomorphism of the tangent bundle.
\end{lemma}
\begin{proof} Our starting point is equation \eqref{eq:laplacianunconjugated}. We will compute $E((1); R)$ at full level and compare the calculation to that of general level. Throughout this calculation we will accumulate error terms, we will always use $E_i$ to denote a smooth bounded function on $(\ln R , \infty) \times B_0$ and $\Omega_i$ to denote a smooth bounded one form on $(\ln R , \infty) \times B_0$.
We write $G = G_{(1)}$ and note
\begin{equation}
G = e^{\frac{n}{2}\tau} (1 + e^{-\tau}J) ,
\end{equation}
and
\begin{equation}\label{eq:inverseexpansion}
G^{-1} = e^{-\frac{n}{2}\tau}(1 + e^{-\tau}J'),
\end{equation}
where $J, J' \in C^\infty( (\ln R , \infty) \times B_0 )$ have bounded derivatives. We will conjugate the terms in equation \eqref{eq:laplacianunconjugated} in turn.
Firstly we calculate
\begin{equation}\label{eq:firstterm}
G \left( - \partial^2_\tau - n\partial_\tau \right) G^{-1} = -\partial_\tau^2 - \left( n + 2 G\partial_\tau(G^{-1}) \right)\partial_\tau - G\left(\partial_\tau^2(G^{-1}) + n \partial_\tau(G^{-1} ) \right).
\end{equation}
From \eqref{eq:inverseexpansion} we can write
\begin{equation}\label{eq:Gderiv1}
\partial_\tau(G^{-1}) = -\frac{n}{2}G^{-1} + e^{-(\frac{n}{2}+1)\tau}\left(\partial_\tau J' - J' \right).
\end{equation}
and taking another derivative gives
\begin{equation}\label{eq:Gderiv2}
\partial_\tau^2(G^{-1}) = \frac{n^2}{4}G^{-1} + e^{-(\frac{n}{2} + 1)\tau}\left(\partial_\tau^2 J' - (n + 2)\partial_\tau J' + (n+1)J' \right).
\end{equation}
Using \eqref{eq:Gderiv1} and \eqref{eq:Gderiv2} in \eqref{eq:firstterm} gives
\begin{align*}
G \left( - \partial^2_\tau - n\partial_\tau \right) G^{-1} &=  -\partial_\tau^2 - 2Ge^{-(\frac{n}{2} + 1)}\left(\partial_\tau J' - J' \right)\partial_\tau \\
&+ \frac{n^2}{4} -  Ge^{-(\frac{n}{2} + 1)\tau}\left(\partial_\tau^2 J' - 2\partial_\tau J' + J'\right).
\end{align*}
This can be written
\begin{equation}
G \left( - \partial^2_\tau - n\partial_\tau \right) G^{-1}  =  -\partial_\tau^2 + \frac{n^2}{4} + e^{-\tau}E_1 \partial_\tau + e^{-\tau} E_2,
\end{equation}
where boundedness of $E_1$ and $E_2$ follows from boundedness of derivatives of $J'$. Similarly the '$f$' term in \eqref{eq:laplacianunconjugated} after conjugation becomes
\begin{equation}
G \left( e^{-\tau}  f \partial_\tau \right) G^{-1} =  e^{-\tau} E_3 \partial_\tau +  e^{-\tau} E_4 .
\end{equation}
Therefore the contribution to the conjugated Laplacian from terms with $\tau$ derivatives is
\begin{equation}
G \left( - \partial^2_\tau - n\partial_\tau - \frac{1}{2}e^{-\tau}  f \partial_\tau \right) G^{-1} =  -\partial_\tau^2 + \frac{n^2}{4} + e^{-\tau}E_5 \partial_\tau + e^{-\tau} E_6 .
\end{equation}
Now we calculate
\begin{align}
G \left(e^{-2 \tau} \Delta_{\g(0)} \right)G^{-1} &= e^{-2\tau}\Delta_{\g(0)} + e^{-3\tau}(1 + e^{-\tau}J) \left( 2(\nabla_{\g(0)}J').\nabla_{\g(0)} + \Delta_{\g(0)}(J')  \right) \\
&= e^{-2\tau}\Delta_{\g(0)} + e^{-3\tau} \Omega_1 . \nabla_{\g(0)}  + e^{-3\tau} E_7 \nonumber,
\end{align}
and
\begin{align}
G \left( e^{-3\tau} a . \nabla_{\g(0)} \right) G^{-1} &= e^{-3\tau} a . \nabla_{\g(0)} + e^{-4\tau}(1 + e^{-\tau}J) a. (\nabla_{\g(0)} J') \\
&=  e^{-3\tau} \Omega_2 . \nabla_{\g(0)} + e^{-4\tau} E_8 \nonumber.
\end{align}
The final term becomes after conjugation
\begin{align}
G \left( e^{-3\tau} \mathrm{div}_{\g(0)} . b . \nabla_{\g(0)} \right) G^{-1}  &= e^{-3\tau}  \mathrm{div}_{\g(0)} . b . \nabla_{\g(0)} + e^{-4 \tau}(1 + e^{-\tau} J)(\nabla_{\g(0)}J') . b . \nabla_{\g(0)} \\
&+e^{-4 \tau}(1 + e^{-\tau} J) ( b .  \nabla_{\g(0)} J') . \nabla_{\g(0)} \nonumber \\
&+ e^{-4\tau}(1 + e^{-\tau} J) \mathrm{div}_{\g(0)} . b . (\nabla_{\g(0)} J') \nonumber \\
&=  e^{-3\tau}  \mathrm{div}_{\g(0)} . b . \nabla_{\g(0)} + e^{-4\tau}\Omega_3 . \nabla_{\g(0)} + e^{-4\tau} E_9. \nonumber
\end{align}
In total therefore we have
\begin{align}
L((1),R) = G\Delta_{X}G^{-1} &=   -\partial_\tau^2 + \frac{n^2}{4} + e^{-2\tau}\Delta_{\g(0)} \\
&+ e^{-\tau} A \partial_\tau + e^{-3\tau} B . \nabla_{\pi_\RR^* \g(0)} \nonumber\\
& + e^{-3\tau} \mathrm{div}_{\pi^*_\RR \g(0) } . C . \nabla_{\pi_\RR^* \g(0)} + e^{-\tau} D \nonumber
\end{align}
for smooth bounded $A$, $B$, $C$ and $D$. The result at a general level $\RR$ holds with $A$ replaced with $A_\RR \equiv \pi_\RR^*(A)$ and similarly for $B$, $C$ and $D$: this follows by repeating the calculation and noting that all the inputs are lifts from level $(1)$. The output error terms are then also lifts and the $\RR$-uniform bounds follow.
\end{proof}

The following Lemma claims that the errors at each level $\RR$ can be treated as perturbations \textit{simultaneously}, by decreasing the size of the neighbourhood at infinity if necessary. A similar Lemma appears in \cite[Lemma 4.1]{PER87}.

\begin{lemma}\label{relbound}
For all $\e > 0$, we can choose $R_0 = R_0(\e)$ large enough so that for all $R > R_0$, and all $f \in C^{\infty}_0 ( \pi_{\RR}^{-1}((\ln R , \infty)_\tau \times B_0 ) )$,
\begin{equation}
 | \langle E( \RR ; R) f , f \rangle | < \e| \langle L_0(\RR ; R) f , f \rangle |.
\end{equation}
\end{lemma}

\begin{proof}
The proof will follow from the identity on $ C^{\infty}_0 ( \pi_{\RR}^{-1}((\ln R , \infty)_\tau \times B_0 ) )$
\begin{equation}\label{eq:relativecontrol}
\langle L_0(\RR;R) f , f \rangle = \| \partial_\tau f \|^2 + \| e^{-\tau} \nabla_{\pi_\RR^* \g(0)}  f \|^2 + n^2/4 \|f\|^2.
\end{equation}
From Lemma \ref{errorform}, $E(\RR;R)$ can be written 
\begin{equation*}
E(\RR;R) = e^{-\tau} A_\RR \partial_\tau + e^{-3\tau} B_\RR . \nabla_{\pi_\RR^* \g(0)}  + e^{-3\tau} \mathrm{div}_{\pi^*_\RR \g(0) } . C_\RR . \nabla_{\pi_\RR^* \g(0)} + e^{-\tau} D_\RR,
\end{equation*}
with $A_\RR , D_\RR$, $\RR$-uniformly bounded functions, $B_\RR$ a vector field with $\RR$-uniform bound and $C_\RR$ representing an endomorphism of the tangent bundle. The norm of $C_\RR$ is $\RR$-uniformly bounded when considered as a function. Each of the terms occurring in $\langle E(\RR;R)f,f \rangle$ can be controlled by \eqref{eq:relativecontrol}, e.g.
\begin{align*}
 |\langle  e^{-3\tau} \mathrm{div}_{\pi^*_\RR \g(0) } . C_\RR . \nabla_{\pi_\RR^* \g(0)} f,f \rangle | &=   |\langle (e^{-\tau} C_\RR) .  e^{-\tau} \nabla_{\pi_\RR^* \g(0)} f, e^{-\tau} \nabla_{\pi_\RR^* \g(0)}  f \rangle |\\
 &\leq \| (e^{-\tau}  C_\RR) .  e^{-\tau} \nabla_{\pi_\RR^* \g(0)} f \| \| e^{-\tau}  \nabla_{\pi_\RR^* \g(0)} f \| \\
 &\leq \| e^{-\tau} C_\RR \|_{op} \: \| e^{-\tau} \nabla_{\pi_\RR^* \g(0)} f \|^2\\
 &\leq  \| e^{-\tau} C_\RR \|_{op}  \langle L_0(\RR;R) f ,f \rangle \\
  &\leq \frac{1}{R_0} \| C_\RR \|_{op}  \langle L_0(\RR;R) f ,f \rangle ,
 \end{align*}
and $\| C_\RR \|_{op}$ is uniformly bounded so by increasing $R_0$ we can sufficiently control the size of this term. The other terms in $\langle E(\RR;R)f,f \rangle$  are estimated similarly.
\end{proof}

\begin{lemma}\label{regularpos}
For any $\eta > 0$ we can choose $R_0 = R_0(\eta)$ large enough so that for all $R > R_0$
\begin{equation}
\Delta_{X(\RR)}\lvert_{C^{\infty}_0 ( \pi_{\RR}^{-1}( (0,1/R) \times B_0  ) )} > n^2 / 4 - \eta 
\end{equation}
uniformly through $\RR$.
\end{lemma}

\begin{proof} 
By Lemma \ref{relbound} we can choose $R_0(\eta)$ so that 
\begin{equation}
 | \langle E( \RR ; R) f , f \rangle | < \frac{4 \eta}{n^2}| \langle L_0(\RR ; R) f , f \rangle | .
\end{equation}
The first two terms in \eqref{eq:L0FORM} are nonnegative and so we have $L_0(\RR ; R) \geq n^2 / 4$. Then
\begin{align*}
\langle L(\RR;R) f,f \rangle &=  \langle L_0(\RR; R) f, f \rangle + \langle E(\RR; R) f, f \rangle \\
&> (1 - 4\eta / n^2) \langle L_0(\RR; R) f, f \rangle \\
&> (n^2/4 - \eta) \langle f,f \rangle,
\end{align*}
which gives the result. 
\end{proof}
\subsection{Eigenfunction estimates}
Let $\chi : (0,\infty) \to [0,1]$ be a smooth cutoff function such that
\begin{equation*}
  \chi(t) =  \left\{
    \begin{array}{rl}
      1 & \text{if } t  \leq 1/2,\\
   
      0 & \text{if } t \geq 1,
    \end{array} \right.
\end{equation*}
and such that $(1 - \chi^2)^{1/2}$ is also smooth. Then let $\chi_k$, $k = 1 ,\ldots , n$ be a set of cutoff functions defined locally on $\delta \bar{X}$ such that in the boundary coordinates corresponding to $\bar{M}_k$
\begin{align*}
  \chi_k(u,z) &= \chi(\frac{3}{2}|u|^2) , \quad k < n \\
  \chi_n(z) &\equiv 1
\end{align*}
and extended by zero to the rest of the boundary $\delta \bar{X}$. Also define scaled versions of the cutoff
\begin{equation}
\chi_R(\rho) = \chi(R \rho )
\end{equation}
which localizes to smaller regions as $R \to \infty$. For $R > R_0$ large enough view $\chi_R$ as a function of $\rho$ on the collar neighbourhood of infinity $(0,\e)_\rho \times \delta \bar{X}$, and extend by zero to a function on the whole of $X$. Define for notational convenience  
\begin{equation} 
\chi_0 \equiv (1 - \sum_{k=1}^n \chi_k^2)^{1/2},
\end{equation}
\begin{equation}
\chi_{R, i,\infty} \equiv \chi_R \chi_i , \quad i = 0, 1 , \ldots , n ,
\end{equation}
\begin{equation}
\chi_{R, K} \equiv (1 - \chi_R^2)^{1/2} .
\end{equation} 
Then the functions
\begin{equation*}
\chi_{R,K} , \chi_{R, i,\infty} , \quad , i = 0 , \ldots , n 
\end{equation*}
form an $R$ parameterized partition of unity for $X$ in the sense of \cite[Definition 3.1]{CFKS}. In particular they are appropriate for application of the IMS localization formula which first appeared explicitly in \cite{SIG82}. Moreover the functions 
\begin{equation*}
\pi_{\RR}^* \chi_{R,K} , \pi_{\RR}^* \chi_{R, i,\infty} , \quad  i = 0 , \ldots , n 
\end{equation*}
form an $R$ parameterized partition of unity for $X(\RR)$. Let $K(R) = X - \rho^{-1}(0,1/2R)$. This is a compact core for $X$ as it is isolated from the boundary.

The bounding below in Lemmas \ref{cusppos} and \ref{regularpos} relied on the localized functions being compactly supported. In order for this to be the case we use an approximation argument following from the well known fact that on a complete Riemannian manifold $(M,g)$, the smooth compactly supported functions $C^{\infty}_0(M)$ are a core for the Laplacian $\Delta_g$. In other words $C^{\infty}_0(M)$ is dense in $L^2(M)$ with respect to the graph norm
\begin{equation}
\| f \|_{\G(\Delta_g)} \equiv \sqrt{ \|f\|^2_{L^2(M)} + \|\Delta_g f\|^2_{L^2(M)} }.
\end{equation}
This result can be found in the paper of Chernoff \cite{CH73}. Now we state and prove the key Lemma.
\begin{lemma}[Eigenfunction estimates]\label{collar}
For any $\e > 0$, there exists an $R = R(\e)$ and a constant $C = C(\e) > 0$ such that if $\phi$ is a normalized eigenfunction of the Laplacian on $X(\RR)$ with eigenvalue $s(n-s) \in [\d(n-\d), n^2/4 - \e]$ then
\begin{equation}
\int_{\pi_{\RR}^{-1}(K(R))} |\phi|^2 dX(\RR) \geq C > 0,
\end{equation}
uniformly through $\RR$.
\end{lemma}
\begin{proof}
Suppose that $\phi$ is a normalized eigenfunction of $\Delta_{\RR}$ on $X(\RR)$ with exceptional eigenvalue $s(n - s) \leq n^2 / 4 - \e$, i.e.
\begin{equation*}
\Delta_{\RR} \phi = s(n-s) \phi , \quad \| \phi \|_{L^{2}(X(\RR))} = 1 .
\end{equation*}
 The IMS localization formula (see \cite[Theorem 3.2]{CFKS} or \cite{SIG82}) tells us how to relate global quantities to local quantities. Using this with the partition of unity $\pi_\RR^* \chi_{R,K}$, $\pi_\RR^* \chi_{R,i,\infty}$ for the quantity $\langle \Delta_\RR \phi , \phi \rangle = s(n-s)$ we have
\begin{align}\label{eq:IMS1} 
s(n-s) &= \langle \Delta_{\RR} (\pi_{\RR}^* \chi_{R,K} ) \phi ,(\pi_{\RR}^* \chi_{R,K}) \phi  \rangle + \sum_{i=0}^{n} \langle \Delta_{\RR} (\pi_{\RR}^* \chi_{R,i,\infty} )\phi , (\pi_{\RR}^* \chi_{R,i,\infty} ) \phi \rangle \\
&- \langle |\nabla_{\RR}(\pi_{\RR}^* \chi_{R,K})|^2 \phi, \phi \rangle - \sum_{i=0}^{n} \langle |\nabla_{\RR}(\pi_{\RR}^* \chi_{R,i,\infty})|^2 \phi, \phi\rangle . \nonumber
\end{align}
The first term is estimated 
\begin{equation}\label{eq:est1}
\langle \Delta_{\RR} ( \pi_{\RR}^* \chi_{R,K}) \phi,(\pi_{\RR}^* \chi_{R,K}) \phi \rangle \: \geq \: \d(n - \d) \| ( \pi_{\RR}^* \chi_{R,K} )\phi \|^2_{L^{2}(X(\RR))},
\end{equation}
by the Patterson-Sullivan description of the bottom of the spectrum. Let $\{\vp_k\}_{k=1}^\infty$ be a sequence in $C^{\infty}_0(X(\RR))$ which goes to $\phi$ in the graph norm. By Lemmas \ref{cusppos} and \ref{regularpos} we can increase $R$ independently of $k$ and $\RR$ so that for all $i$
\begin{equation*}
\langle \Delta_{\RR} ( \pi_{\RR}^* \chi_{R,i,\infty}) \vp_k , ( \pi_{\RR}^* \chi_{R,i,\infty}) \vp_k \rangle \: \geq \: (n^2 / 4 - \e /4) \| ( \pi_{\RR}^* \chi_{R,i,\infty}) \vp_k \|^2_{L^{2}(X(\RR))}.
\end{equation*}
Taking the limit in $k$ to get the corresponding statement for $\phi$ and summing over $i$ we have
\begin{equation}\label{eq:est2}
 \sum_{i=0}^{n} \langle \Delta_{\RR} ( \pi_{\RR}^* \chi_{R,i,\infty}) \phi ,( \pi_{\RR}^* \chi_{R,i,\infty})\phi\rangle \: \geq \: (n^2 / 4 - \e /4) \| (\pi_{\RR}^* \chi_R) \phi \|^2_{L^{2}(X(\RR))} .
\end{equation}
The remaining terms in \eqref{eq:IMS1} can be estimated by noting
\begin{equation*}
\nabla_{\RR}(\pi_{\RR}^* \chi_{R,i,\infty}) = \pi_{\RR}^* (\nabla_{(1)}\chi_{R,i,\infty}),
\end{equation*}
so using the product rule for the gradient, and that the projection is a local isometry,
\begin{align*}
|\nabla_{\RR}(\pi_{\RR}^* \chi_{R,i,\infty})|^2 &=  \pi_{\RR}^* \left(  \chi_R^2 |\nabla_{(1)}\chi_i|^2 + \chi_i^2|\nabla_{(1)}\chi_R|^2 + 2 \chi_R \chi_i \langle \nabla_{(1)}\chi_i , \nabla_{(1)}\chi_R \rangle \right) \\
&=  \pi_{\RR}^* \left(  \chi_R^2 |\nabla_{(1)}\chi_i|^2 + \chi_i^2|\nabla_{(1)}\chi_R|^2  \right), 
\end{align*}
where the last term on the first line vanished due to the form of the metric. Summing over $i$, and using the estimates \eqref{eq:est1} and \eqref{eq:est2} in \eqref{eq:IMS1} we have
\begin{align}
s(n-s)\: &\geq \: \d(n - \d) \| (\pi_{\RR}^* \chi_{R,K}) \phi \|^2_{L^{2}(X(\RR))} + (n^2 / 4 -\e / 4) \| (\pi_{\RR}^* \chi_R) \phi \|^2_{L^{2}(X(\RR))} \\
&- \langle \pi_{\RR}^* (|\nabla_{(1)}\chi_R|^2 + |\nabla_{(1)}\chi_{R,K}|^2) \phi, \phi \rangle - \langle \pi_{\RR}^* (\chi_R^2 \sum_{i=0}^n |\nabla_{(1)}\chi_i|^2 ) \phi, \phi \rangle  . \nonumber 
\end{align}
The terms on the second line will be estimated by $L^\infty$ norms which are preserved under $\pi_{\RR}^*$. Now we observe that $ \chi_R^2 |\nabla_{(1)}\chi_i |^2 $ is supported only for $\rho \leq 1/R$, and there
 \begin{equation*} 
\chi_R^2 |\nabla_{(1)}\chi_i |_g^2 \leq  |\nabla_{(1)}\chi_i |_g^2 = \rho^2 |\nabla_{u, euc} \chi_i |_{euc}^2 \leq 1/R^2 |\nabla_{u, euc} \chi_i |_{euc}^2 ,
 \end{equation*}
 where $\nabla_{u, euc}$, $|.|_{euc}$ refer to the Euclidean gradient and metric for the $u$ coordinate in the regions $\bar{M}_k$. Therefore we can increase $R$ so that 
 \begin{equation*}
 \|  \chi_R^2 \sum_{i=0}^n |\nabla_{(1)}\chi_i|^2 \|_\infty < \e / 4 .
 \end{equation*}
Incorporating these estimates we have, letting
\begin{equation*}
F_R =  (s(n-s) - \d(n - \d)) \chi^2_{R,K} + |\nabla_{(1)}\chi_R|^2 + |\nabla_{(1)}\chi_{R,K}|^2 ,
\end{equation*}
then
\begin{align*}
\| F_R^{1/2} \|_\infty^2 \| \phi \|^2_{|L^{2}(\pi_{\RR}^{-1}(K(R)))} &\geq  \|  (\pi_{\RR}^* F_R )^{1/2} \phi \|^2_{L^{2}(X(\RR))}\\
&\geq (n^2 / 4 - \e /2- s(n-s) ) \|  \phi \|^2_{L^{2}(X(\RR) - \pi_{\RR}^{-1}(K(R)))}.
\end{align*} 
Now we note that  $\| F_R^{1/2} \|_\infty$ is uniformly bounded as $R \to \infty$. This follows from $s(n-s) < n^2 / 4$, and for example
\begin{equation*}
|\nabla_{(1)}\chi_R|_g^2 = |\rho^2 R \chi' |_g^2 = \rho^2 R^2 |\chi'|^2 \leq |\chi'|^2 ,
\end{equation*}
where $'$ denotes a derivative. Then we have shown
\begin{align*}
1  &=  \| \phi \|^2_{|L^{2}(\pi_{\RR}^{-1}(K(R)))}  +   \|  \phi \|^2_{L^{2}(X(\RR) - \pi_{\RR}^{-1}(K(R)))} \\
&\leq \left(1 +  \frac{\| F_R^{1/2} \|^2_\infty}{(n^2 / 4 - \e / 2  - s(n-s)  )}\right)\| \phi \|^2_{|L^{2}(\pi_{\RR}^{-1}(K(R)))} \\
&\leq \left(1 +  \frac{\| F_R^{1/2} \|^2_\infty}{\e/2}\right)\| \phi \|^2_{|L^{2}(\pi_{\RR}^{-1}(K(R)))} \\
&\leq \left(1 +  C_0(\e) \right)\|\phi\|^2_{|L^{2}(\pi_{\RR}^{-1}(K(R)))} \\
&= \left(1 +  C_0(\e) \right) \int_{\pi_{\RR}^{-1}(K(R))} |\phi |^2 dX(\RR)
\end{align*}
which establishes the result by taking 
\begin{equation*} 
C(\e) = 1/(1+ C_0(\e)) > 0.
\end{equation*}
\end{proof}
\section{Analytic Preparations}\label{analysis}
\subsection{Estimates for terms appearing in the trace formula}
Let $o$ denote the point corresponding to $(0 , \ldots , 0 , 1 )$ in the hyperboloid model for $\H^{n+1}$. We write $G = \SO^{0}(n + 1, 1 )$, $K = \SO(n+1)$ the maximal compact subgroup. $T$ will be some generating parameter throughout the rest of the paper. For the details on spherical functions the reader can see \cite{HEL08}. For the representation theory we refer to \cite{KNA86}. All further ideas in this section are due to Sarnak and Xue \cite{SX91}.

The lattice point count relates to harmonic analysis by consideration of the function $\chi_T : G \to \R$
\begin{equation}
\chi_T ( g ) =  \left\{
    \begin{array}{rl}
      1 & \text{if } d(o , g(o) ) \leq  T ,\\
   
      0 & \text{if } d(o, g(o))  > T .
    \end{array} \right.
\end{equation}

For any $\lambda = s(n - s) \in (0,n^2/4)$ we can consider the associated complimentary series representation $\pi_s$.  This contains a normalized spherical ($K$-invariant) vector $v$. From this data we construct the spherical function
\begin{equation}
\phi_s(g) \equiv \langle \pi_s(g) v , v \rangle .
\end{equation}
We let $f_s (g) = \chi_T(g) \phi_s (g)$, and $F_s = f_s * \overline {f_s } $ with $\overline {f_s} (g) = \overline{ f_s(g^{-1} ) }$. Then following \cite[Lemma 2.1]{SX91}  we have $F_s \in C_0(K \backslash G / K)$ and
\begin{equation}\label{eq:upper}
F_s(g) \ll   \left\{ \begin{array}{rl}
      e^{2(s - n/2)T } e^{-\frac{n}{2}d(o, g(o))} & \text{if } d(o , g(o) ) \leq  2T ,\\
   
      0 & \text{if } d(o, g(o))  > 2T .
    \end{array}\right.
\end{equation}
The implied constant can be taken uniformly for $s \in I \subset (n/2  , \delta ]$, for $I$ a closed interval.
As $F_s$ is $K$-biinvariant there is an associated spherical transform of $F_s$, defined for $\lambda_t = t(n-t)$
\begin{equation}
\hat{F_s}(\lambda_t) = \int_G F_s(g) \overline{\phi_t (g)} dG .
\end{equation}
In fact, the spherical transform is a $*$-homomorphism so we have, evaluating at $\lambda = \lambda_s$ 
\begin{equation} 
\hat{F_s}(\lambda) = |f_s(\lambda) |^2 =  \left( \int_G \chi_T(g) |\phi_s(g)|^2 dG \right)^2 ,
\end{equation}
and using the property
\begin{equation}
\phi_s (\exp tX) \gg e^{(s - n) t}
\end{equation}
we have 
\begin{equation}\label{eq:lower}
\hat{F_s}(\lambda)  =  \left( \int_G \chi_T(g) |\phi_s(g)|^2 dG \right)^2 \gg \left( \int_0^T e^{2(s-n)t}e^{nt} dt \right)^2 \gg e^{4(s- n/2)T} .
\end{equation}
Moreover, as before the implied constant can be chosen uniformly for $s$ in compact $I  \subset (n/2  , \delta ]$.

\subsection{Lattice point count}
Now we estimate the quantity
\begin{equation}
N(\Gamma(\RR) , T ) \equiv | \{ \g \in \G(\RR) : d(o, \g o) \leq T \} |.
\end{equation}
In \cite{SX91} Sarnak and Xue conjectured that
\begin{equation}
N(\Gamma(\RR) , T ) \ll_\e \frac{e^{nT(1+\e)}}{[ \G : \G(\RR)]} + e^{nT/2}
\end{equation}
for $\G$ an arithmetic lattice in $\SO(n+1,1)$. They established this result for $n = 1,2$ by a direct counting argument.

For $n \geq 3$ we will rely on a result of Kelmer and Silberman from \cite{KS10}. This uses the spectral theory at the cofinite level (for $\G$). The best known spectral gap when $n \geq 3$ is given by Theorem \ref{BC} (a result of Bergeron and Clozel from \cite{BC12}). This tells us that if $s > n/2$ and $s(n-s)$ is a nonzero eigenvalue for $\H^{n+1} / \Gamma(\RR)$ then $s \leq n - 1$.

The consequence for the lattice point count in $\SO(n+1,1)$, $n \geq 3$ is
\begin{equation}
N(\Gamma(\RR) , T ) \ll \frac{e^{nT}}{[ \G : \G(\RR)]} + e^{(n - 1)T}
\end{equation}
uniformly in $T , \RR$. This result appears in \cite[Theorem 2]{KS10}. Using the estimate on the size of the factor group from Lemma \ref{generalfinite}, we have the following bound for the lattice point count.
\begin{lemma}[Lattice point count]\label{lattice} For any $\e > 0$ and $n \geq 2$ we have
\begin{equation}
N(\Gamma(\RR) , T ) \ll_\e \frac{e^{nT(1+\e)}}{|\O_F / \RR|^{(n+2)(n+1)/2}} + e^{(n - 1 )T}.
\end{equation}
\end{lemma}    
\section{Proof of Main Theorem}\label{proof}
Let $I = [a ,\delta]$ a closed interval for some $a > n/2$. Replace $\Lambda$ with the $\Lambda_1$ of Section \ref{algebra} if necessary. We aim to apply the pre-trace formula to the automorphic kernel on $\H^{n+1} \times \H^{n+1}$ corresponding to $F_s$ at level $\RR$ i.e.
\begin{equation}
K_\RR(x_1, x_2) = \sum_{\g \in \Lambda(\RR)} F_s(g_{x_1}^{-1} \g g_{x_2}) ,
\end{equation}
where we write $g_{x}$ for any group element such that $g_x (o) = x$.
We have for the spectral decomposition of the automorphic kernel, as in \cite[Proposition 5.2]{GAM02}, 
\begin{equation}
K_\RR(x, x) = \sum_{\lambda_{j,\RR} < n^2 / 4}  \hat{F_s}(\lambda_{j,\RR})|\psi_j(x)|^2 + \mathcal{E} ,
\end{equation}
where $\mathcal{E}$ is some \textit{nonnegative} contribution from the continuous spectrum. The $\lambda_{j,\RR}$ are the eigenvalues of the Laplacian on $X(\RR)$ below $n^2/4$, counted with multiplicities. The $\psi_i$ are the corresponding (lifted) eigenfunctions. There are only finitely many such eigenfunctions by the work of Lax and Phillips \cite{LP82}.

We can now apply the eigenfunction estimates (Lemma \ref{collar}) to find a compact part $\K \subset X$ such that for all $s_i \in I$
\begin{equation} 
\int_{\pi_\RR^{-1}(\K)} |\psi_j(x)|^2 dX(\RR) \geq C > 0 
\end{equation}
uniformly through $s_i \in I$ and $\RR$. This implies
\begin{align*}
\int_{\pi_\RR^{-1}(\K)} K_\RR(x, x) dX(\RR) &\geq  \sum_{\lambda_{j,\RR} < n^2 / 4} \hat{F_s}(\lambda_{j,\RR}) \int_{\pi_\RR^{-1}(\K)} |\psi_j(x)|^2 \\
&\geq \sum_{\lambda_{j,\RR} : s_j \in I} \hat{F_s}(\lambda_{j,\RR}) \int_{\pi_\RR^{-1}(\K)} |\psi_j(x)|^2 \\
&\geq C  \sum_{\lambda_{j,\RR} : s_j \in I} \hat{F_s}(\lambda_{j,\RR}).
\end{align*}
In particular, for any $s \in I$, if $\lambda = s(n-s)$ appears as an eigenvalue of $\Delta_{X(\RR)}$
\begin{equation}\label{eq:lowertrace}
\int_{\pi_\RR^{-1}(\K)} K_\RR(x,x) dX(\RR) \gg \hat{F_s}(\lambda) \gg  e^{4(s- n/2)T} m(\lambda, \RR),
\end{equation}
where $m(\lambda, \RR)$ is the multiplicity of the eigenvalue and the implied constant is uniform through all $s \in I$. The last inequality is a result of the estimate for the spherically transformed kernel in \eqref{eq:lower}.

On the other hand
\begin{align*}
\int_{\pi_\RR^{-1}(\K)} K_\RR(x,x) dX(\RR) &= \sum_{\g \in \Lambda(\RR)} \int_{\pi_\RR^{-1}(\K)} F_s (g_x^{-1} \g g_x ) dX(\RR) \\
&= \sum_{\g \in \Lambda(\RR)} \sum_{l \in \Lambda / \Lambda(\RR) } \int_{\K} F_s (g_x^{-1} l^{-1} \g l g_x ) dX \\
&\ll |\O_F / \RR|^{(n+2)(n+1)/2}  \sum_{\g \in \Lambda(\RR)}  \int_{\K} F_s (g_x^{-1}  \g  g_x ) dX \\
&\ll  |\O_F / \RR|^{(n+2)(n+1)/2}  \sum_{\g \in \Gamma(\RR)}  \int_{\K} F_s (g_x^{-1}  \g  g_x ) dX .
\end{align*}
The penultimate inequality uses the bound on the size of the group $\Lambda / \Lambda(\RR)$ given in Lemma \ref{generalfinite}. The last inequality is a result of the rather crude observation that $\Lambda(\RR) \subset \Gamma(\RR)$. Using the upper bound \eqref{eq:upper} we have
\begin{align}
\sum_{\g \in \Gamma(\RR)}  \int_{\K} F_s (g_x^{-1}  \g  g_x ) dX &\ll e^{2(s - n/2)T }  \sum_{\g \in \Gamma(\RR)} \int_{x \in \K : d(x,\g x) \leq 2T}   e^{-\frac{n}{2}d(x, \g x)} dX.
\end{align}
As $\K$ is compact there exists $R$ such that $\K \subset B(o,R)$ in $\H^{n+1}$. This gives
\begin{equation*}
 d(o , \g o ) \leq d(o, x ) + d(x, \g x ) + d ( \g x , \g o ) \leq 2R + d(x , \g x),
\end{equation*} 
and 
\begin{equation*}
 d(x , \g x) \leq 2R + d(o , \g o).
 \end{equation*}
Then 
\begin{align*}
 \sum_{\g \in \Gamma(\RR)} \int_{x \in \K : d(x,\g x) \leq 2T}   e^{-\frac{n}{2}d(x, \g x)} dX &\ll  \sum_{\g \in \Gamma(\RR) : d(o,\g o) \leq 2T + 2R }  e^{-\frac{n}{2}d(o, \g o)} dX \\
&\ll \int_{0}^{2T + 2R}  e^{-\frac{n}{2}t} N(\G(\RR) , t) dt ,
\end{align*}
by integrating by parts. By using the lattice point bound from Lemma \ref{lattice} we estimate
\begin{align*}
\int_{0}^{2T + 2R}  e^{-\frac{n}{2}t} N(\G(\RR) , t) dt  &\ll_\e \int_{0}^{2T + 2R}  \frac{e^{nt(1/2+\e)}}{|\O_F / \RR|^{(n+2)(n+1)/2}} + e^{(n/2 - 1)t} \\
& \ll_\e \frac{e^{nT(1+2\e)}}{|\O_F / \RR|^{(n+2)(n+1)/2}} + e^{(n - 2)T} .
\end{align*}
Gathering together we have the upper bound 
\begin{equation}\label{eq:uppertrace}
\int_{\pi_\RR^{-1}(\K)} K_\RR(x,x) dX(\RR) \ll_\e e^{2(s - n/2)T } \left( e^{nT(1+2\e)} + |\O_F / \RR|^{(n+2)(n+1)/2}e^{(n - 2)T} \right),
\end{equation}
so that we now have upper and lower bounds for the partial trace. Equations \eqref{eq:uppertrace} and \eqref{eq:lowertrace} give
\begin{equation*}
e^{2(s- n/2)T} m(\lambda, \RR) \ll_\e e^{nT(1+2\e)} + |\O_F / \RR|^{(n+2)(n+1)/2}e^{(n - 2)T}.
\end{equation*}
This is the keystone of the proof. Our previous work on the multiplicities and the lattice point count will now play together to forbid certain values of $s$. Using Lemma \ref{generalfinite} to estimate $m(\lambda,\RR)$ we now have for any $\e > 0$
\begin{equation}
e^{2(s- n/2)T} |\O_F / \RR|^{n-1} \ll_\e e^{nT(1+2\e)} + |\O_F / \RR|^{(n+2)(n+1)/2}e^{(n - 2)T}.
\end{equation}
Taking $T \approx ((n+1)(n+2)/4) \ln |\O_F / \RR| $ we have for all $\e > 0$
\begin{equation}
|\O_F / \RR|^{(2s - n)(n+1)(n+2)/4 + n - 1} \ll_\e  |\O_F / \RR|^{(1 + 2\e)n(n+1)(n+2) / 4}
\end{equation}
as $|\O_F / \RR| \to \infty$, which can only be true if
\begin{equation*}
s \leq s^0_n \equiv n - \frac{2(n-1)}{(n+1)(n+2)}.
\end{equation*}
Recall that at the start we were free to choose an interval $I_s$ throughout which we had uniformity. Going back and choosing
\begin{equation*}
I = [s^0_n, \delta], 
\end{equation*}
which makes sense as long as $\delta > s^0_n$, we have proved Theorem \ref{maintheorem}. Corollary \ref{maincorollary} follows directly.

\section{Construction of thin groups with thick limit sets}\label{construct}
 Our main theorem (Theorem \ref{maintheorem}) gives a quantitative spectral gap for infinite covolume subgroups of hyperbolic isometries, provided the Hausdorff dimension of the limit set is large enough (and some other conditions are met). In other words, the 'bass note' of the quotient manifold must be low. In order to find such a group, we must remain safe from Doyle's pretty result \cite{D88}, which states that there is a universal upper bound for the Hausdorff dimension of the limit sets of Schottky groups in $\Isom(\H^3)$.
 
The following procedure for constructing thin subgroups of hyperbolic arithmetic lattices with arbitrarily large Hausdorff dimension is due to McMullen \cite{MCMEMAIL}. He has communicated a method which takes as input a compact $n+1$ dimensional hyperbolic manifold $M$ which arises as an arithmetic quotient $\H^{n+1} / \G$ and which has an embedded totally geodesic hypersurface $S$ such that $[S] \neq 0$ in $H_n(M, \Z)$. The output is a set of subgroups of $\G$ which are geometrically finite, infinite index and have Hausdorff dimensions of the limit set arbitrarily close to $n$. This yields groups for which our main theorem (Theorem \ref{maintheorem}) applies. This construction is similar in nature to the construction of Gamburd in the last section of \cite{GAM02}. Slightly more care is required in higher dimension, for example because of the inequivalence of geometrical finiteness and finite generation.

The required input arises in the work of Millson \cite{MIL76}, whose development we recount now. We take as our field $F = \Q(\sqrt{p})$ with $p$ prime and consider the quadratic form
\begin{equation}
q(x_1 , \ldots , x_{n+2} )  = x_1^2 + x_2^2 + \ldots x_{n+1}^2 - \sqrt{p} \: x_{n+2}^2 .
\end{equation}
We then take $\G_0$ to be the subgroup of $\mathrm{GL}_{n+2}(\O_F)$ which preserves $q$. The quadratic form $q$ is conjugate to $\mathrm{diag}(1 , \ldots , 1 , -1)$ by a diagonal matrix over $\R$, via conjugation with the same element we can realize $\G_0$ as a group of isometries of the hyperbolic plane $\H^{n+1}$. We let $\G$ be the orientation preserving subgroup of $\G_0$. As we have remarked in a previous section, $\G$ is discrete. Moreover $\G$ is uniform (cocompact) as follows. There are no rational zeros of $q$, and $\G$ contains no nontrivial unipotent elements (if $\g \in \G$ is unipotent then so is its Galois conjugate, which is an element of a definite orthogonal group). The work of Mostow and Tamagawa \cite{MT62} then gives uniformity of $\G$. It is clear that $\G$ is geometrically finite.
 
We next consider the congruence subgroups $\G(\P)$ for $\P$ prime in $F$. As in \cite{MIL76} for $\P$ of large enough norm $\G(\P)$ is torsion-free and the quotient $\H^{n+1} / \G(\P)$ is a compact manifold of constant negative curvature. Millson considers the involution $\iota$ in $\G$ which is the reflection in the plane $x_1 = 0$. As $\G(\P)$ is normal in $\G$ then $\iota$ normalizes $\G(\P)$ and hence descends to an involution on $\H^{n+1} / \G(\P)$. For $\P$ not of norm two and such that $\G(\P)$ is torsion free we have then a manifold $M(\P) = \H^{n+1} / \G(\P)$ with an involution $\iota$ whose fixed point set contains a constantly negatively curved, orientable, codimension one submanifold $S(\P)$. Millson goes on to show that by passing to congruence subgroups 
\begin{equation}
\G(\P , \P') \equiv \{ \g \in \G  : \g \equiv I \mod \P , \: \g \equiv I \mod \P' \}
\end{equation}
with $\P$ and $\P'$ of large enough norm we obtain a manifold $M$ = $\H^{n+1} / \G(\P , \P')$ and a corresponding embedded totally geodesic hypersurface $S$ such that $[S] \neq 0 \in H_n(M,\Q)$. We have $\pi_1(M) = \G(\P , \P')$ a geometrically finite and finite index subgroup of a uniform arithmetic lattice.

We are now ready to apply McMullen's construction \cite{MCMEMAIL}. 

Let $\doublecap : H_k(M , \Z) \times H_{n+1 - k}(M, \Z) \to H_0(M,Z) \cong \Z $ denote intersection number and consider the homomorphism
\begin{equation}
\phi : \pi_1(M) \to \Z , \quad \phi(\g) = [S] \doublecap [\g] ,
\end{equation}
by using the usual map $\pi_1(M) \to H_1(M,\Z)$. Poincar{\'e} duality gives that $\phi$ is defined and nonzero (the reader can see Hatcher \cite{HAT02} for details), as $S$ is oriented and embedded $\phi$ is moreover onto. The kernel $\ker \phi$ consists of classes with generic representatives which cross $S$ the same number of times in each direction relative to a fixed orientation of $S$. 

We can cut $M$ open along $S$ to give a connected manifold with boundary $M_{cut}$. The boundary $\delta M_{cut}$ consists of two connected components $S_1$ and $S_2$ which are both isometric to $S$. Recall that $S$ was obtained as a quotient of a totally geodesic hyperplane $\Pi \subset \H^{n+1}$. In this case it is known that $\pi_1(M) = \pi_1(M_{cut})*_{\pi_1(S)}$, the HNN extension of $\pi_1(M_{cut})$ with respect to the maps $i_{j*}  : \pi_1(S_j) \to \pi_1(M_{cut})$ induced from the inclusions of the boundary components of $M_{cut}$. Formally the HNN extension is generated by $\pi_1(M_{cut})$ and an element $t$ subject to the relations 
\begin{equation}
t i_1(\g) t^{-1} = i_2(\g) ,
\end{equation}
concretely this element $t$ can be chosen to be anything with $\phi(t) = 1$. Fix such a $t$.

To obtain manifolds with many repeated sections we consider the collection of hyperplanes 
\begin{equation}
\mathcal{C}_k =  \{ \g \Pi : \g \in \phi^{-1}(0) = \ker\phi \}  \cup \{ \g \Pi : \g \in \phi^{-1}(k) \}. 
\end{equation}Distinct hyperplanes in this collection do not intersect, and there is a connected component of $\H^{n+1} - \mathcal{C}_k$ which meets all of the hyperplanes in $\mathcal{C}_k$, we write $R_k$ for the closure of this connected component. We let $\G_k = \mathrm{Stab}_{\pi_1(M)} ( R_k)$ so that for $k = 1$ we have that $\G_1  = \pi_1(M_{cut})$ and $R_1$ a universal cover for $M_{cut}$, i.e. $M_{cut} = R_1 / \G_1$.

The quotient $R_k / \G_k$ is then the manifold with boundary which is obtained by gluing copies $M_{cut}^{i}$, $i = 1 , \ldots , k$ together by identifying one boundary component of $M_{cut}^i$ with the opposite boundary component of $M_{cut}^{i+1}$, leaving two boundary components unglued (coming from the first and last copies of $M_{cut}$).

We now note some properties of $\G_k$. As $\G_k$ is contained in the kernel of $\phi$ it is of infinite index in $\pi_1(M)$. As $R_k$ is invariant under $\G_k$ the limit set $L(\G_k)$ of $\G_k$ must lie between the boundary of $\Pi$ and the boundary of $t^k\Pi$. This implies that the quotient $CL(\G_k) / \G_k$ of the closed convex hull of $L(\G_k)$ by $\G_k$ has finite volume (in fact it is compact). Also noting that $\G_k$ is finitely generated implies that $\G_k$ is geometrically finite by a result of Bowditch \cite{BOW93}. 

We have constructed $\G_k$, an infinite index geometrically finite subgroup of an arithmetic lattice in $\SO(n+1 , 1)$. It remains to show that by making $k$ large enough we can force the Hausdorff dimension of the limit set $\delta(L(\G_k))$ to be as close to $n$ as we like. Equivalently by the work of Sullivan \cite{SUL82} we can show that there are arbitrarily small eigenvalues of the Laplacian on $X_k \equiv \H^{n+1} / \G_k$ for $k$ large. This will follow by finding functions $u_k$ with small Rayleigh quotient
\begin{equation} \label{eq:Ray}
\frac{ \int_{X_k} \| \nabla u_k \|^2 d X_k}{ \int_{X_k} |u_k|^2 dX_k}.
\end{equation}
We have $R_k  /  \G_k$ a closed subset of $X_k$, consisting of $k$ glued copies $M^i_{cut}$ of $M_{cut}$. Let $f_+ \in C^{\infty}(M_{cut})$ be such that $f_{+}\lvert_{S_1} \equiv 0$ and $f_{+}\lvert_{S_2} \equiv 1$ and which is locally constant in a neighbourhood of the boundary. Now we define for $k \geq 2$
\begin{equation}  
u_k  =\left\{ \begin{array}{rl}
      f_+ & \text{on } M^1_{cut}  \\
      1 & \text{on }  M^{i}_{cut} , \: i = 2 ,\ldots , k-1 \\
      (1 - f_+)& \text{on } M^k_{cut} \\
      0 & \text{on } X_k - R_k / \G_k  . \\
    \end{array}\right. 
\end{equation}
Calculation of the Rayleigh quotient \eqref{eq:Ray} gives
\begin{equation}
\frac{ \int_{X_k} \| \nabla u_k \|^2 d X_k}{ \int_{X_k} |u_k|^2 dX_k} = \frac{ 2 \int_{M_{cut}} \| \nabla f_+ \|^2 d M_{cut}}{ \int_{M_{cut}} (f_+)^2 + (1-f_+)^2  d M_{cut} + (k-2)\mathrm{vol}(M_{cut}) }
\end{equation}
which tends to $0$ as $k \to \infty$ (the only term depending on $k$ is the last summand in the denominator). This completes the construction.


\end{document}